\documentclass[11pt]{article}

\usepackage{graphicx}
\usepackage{url}
\usepackage{amsfonts}
\usepackage{amsmath}
\usepackage{amsthm}
\usepackage{amssymb}
\usepackage{url}
\usepackage{subfig}
\usepackage{hyperref}
\usepackage{fullpage}

\allowdisplaybreaks

\newtheorem{theorem}{Theorem}

\newtheorem{lemma}[theorem]{Lemma}

\providecommand{\seg}[1]{|#1|}
\def\DEF#1{\emph{#1}}

\title{The intersection graph of the disks with diameters\\
the sides of a convex $n$-gon\footnote{The partial solution for the convex pentagon appeared at the XVI Spanish Meeting on Computational Geometry, Barcelona, Spain, 2015.}}

\author{
Clemens Huemer\thanks{Departament de Matem\`atica Aplicada IV, Universitat Polit\`ecnica de Catalunya, Barcelona, Spain.
{\tt clemens.huemer@upc.edu}.} 
\and
Pablo P\'erez-Lantero\thanks{Departamento de Matem\'atica y Ciencia de la Computaci\'on, Universidad de Santiago, Santiago, Chile.
{\tt pablo.perez.l@usach.cl}.}
}

\begin{document}
\maketitle

\begin{abstract}
Given a convex polygon of $n$ sides, one can draw $n$ disks (called side disks)
where each disk has a different side of the polygon as diameter 
and the midpoint of the side as its center. The intersection graph
of such disks is the undirected graph with vertices the $n$ disks and
two disks are adjacent if and only if they have a point in common.
We prove that for every convex polygon this graph is planar.
Particularly, for $n=5$, this shows that for any convex pentagon there are two disks
among the five side disks that do not intersect, which means that $K_5$ is never the intersection graph of such five disks.
For $n=6$, we then have that for any convex hexagon the intersection graph of the side disks 
does not contain $K_{3,3}$ as subgraph.
\end{abstract}

\section{Introduction}

Let $P_n$ be a convex polygon of $n$ sides denoted $s_0,s_1,\ldots,s_{n-1}$ counter-clockwise. 
For each side $s_i$, let $D_i$ denote the disk with diameter the length of $s_i$ and center
the midpoint of $s_i$. Since $D_i$ is constructed on the side $s_i$ of $P_n$, we say
that $D_i$ is a \DEF{side disk} of $P_n$. The \DEF{intersection graph} of the side disks
$D_0,D_1,\ldots,D_{n-1}$ is the undirected graph $G=\langle V,E\rangle$, where 
$V=\{D_0,D_1,\ldots,D_{n-1}\}$ and $\{D_i,D_j\}\in E$ if and only if 
the intersection of $D_i$ and $D_j$ is not empty. We prove in this paper that for any convex
polygon the intersection graph of the side disks is planar.


Results on disjoint and intersecting disks in the plane are among the most classical ones in convex geometry. Helly's theorem, see e.g.~\cite{wenger97}, when stated for disks in the plane, tells us that if any three of a given family of $n$ disks intersect, then all $n$ disks intersect. A theorem of Danzer~\cite{danzer} says that if any two of a given
family of $n$ disks intersect, then there exists a set of four points which intersects each disk. We refer to the survey of Wenger for related results~\cite{wenger97}. 
Intersections of disks have also been considered in the context of intersection graphs: each disk represents a vertex of the graph and two vertices are adjacent if and only if the corresponding disks intersect. By the Koebe-Andreev-Thurston theorem~\cite{koebe}, every planar graph is an intersection graph of disks, where every pair of intersecting disks have only one point in common, that is, they are tangent.
In our problem, the disks 
are in special position and we prove that their intersection graph is planar.

Results relating convex polygons and disks (or circles) can also be found in the literature. Given a polygon with $n$ vertices 
$A_1,A_2,\ldots,A_n$, a sequence $D_1,D_2,D_3,\dots$ of disks can be built as follows: $D_1$
is inscribed in the angle $A_1$, $D_2$ is the smallest one inscribed in the angle $A_2$ and tangent to $D_1$,
$D_3$ is the smallest one inscribed in the angle $A_3$ and tangent to $D_2$, etc. The Money-Coutts
theorem~\cite{evelyn1974seven}, also known as the {\em six circles theorem}, 
states that for every triangle $A_1,A_2,A_3$ there exists such a sequence of disks that is 6-periodic.
In general, if at least one of the tangency points of the initial disk $D_1$ lies on a side of the triangle,
then the sequence of disks is eventually 6-periodic~\cite{ivanov2013six}.
A similar 6-periodic sequence of disks can be built with respect to three given circles in general
position instead of a triangle~\cite{tyrrell1971theorem}. In this scenario, Tabachnikov~\cite{tabachnikov2000going}
described a class of convex $n$-gons for which such a sequence is $2n$-periodic. 
Troubetzkoy~\cite{troubetzkoy2000circles} showed that for parallelograms such sequences are preperiodic, with eventual period 4, and that they are chaotic for some class of quadrilaterals. Also well known
are Malfatti's problem which given a triangle asks for three non-overlapping disks of maximum
total area contained in the triangle~\cite{zalgaller1994solution}, and the 
{\em seven circles theorem}~\cite{evelyn1974seven} that given a closed chain of six circles all tangent to, and enclosed by, 
a seventh circle and each tangent to its two adjacent ones, states that the three lines drawn between the opposite pairs of tangency points on the seventh circle are concurrent.
  
The problem studied here, posed in~\cite{irracional} for the case of a pentagon, turned out to be non-trivial to solve. To give some insight for this particular case of a pentagon, let us consider the example in Figure~\ref{fig:approach}. The two disks corresponding to sides $AB$ and $CD$ do not intersect. This is equivalent to saying that the distance between the midpoints $M_{AB}$ and $M_{CD}$ of the segments $AB$ and $CD$, respectively, is larger than the sum of the radii of the two disks, equal to half of the sum of the lengths of $AB$ and $CD$. Thus, a first natural approach is trying to prove that the sum of the five distances between the midpoints (that is, the dotted edges in Figure~\ref{fig:approach}) is bigger than the perimeter of the pentagon. But this is not always the case: for example, consider the convex pentagon with vertices at coordinates $(1,9)$, $(0,3)$, $(0,-3)$, $(1,-9)$, and $(60,0)$. Another quite natural approach is to connect an interior point $P$ of the pentagon with the five vertices and consider the angles at $P$.  It follows from Thales' Theorem that $P$ lies in the disk with, say, diameter $AE$ if and only if the angle $\angle APE$ is at least $\pi/2$. In Figure~\ref{fig:approach2}, $P$ lies outside the disks with diameters $AB$, $BC$, and $DE$. Clearly, no point $P$ lies in more than three of the five disks, since otherwise the five angles around $P$ would sum more than $2\pi.$ One could then use a fractional version of Helly's theorem (Theorem 12 in~\cite{wenger97}), which states that if among all the 10 triples of the five disks, more than 6 triples have a point in common, then there exists a point contained in 4 disks. We conclude that there are at least 4 triples of disks without a common intersection. However, it remained elusive to us to solve this particualar case with a Helly-type approach.

\begin{figure}[t]
	\centering
	\subfloat[]{
		\includegraphics[scale=0.8,page=9]{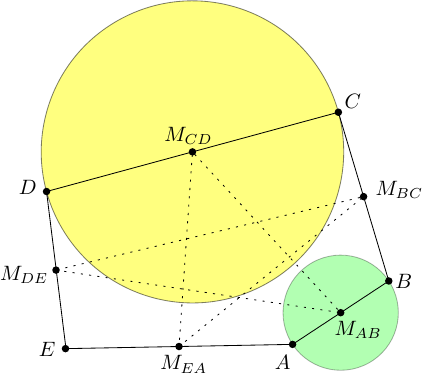}
		\label{fig:approach}
	}
	\subfloat[]{
		\includegraphics[scale=0.8,page=10]{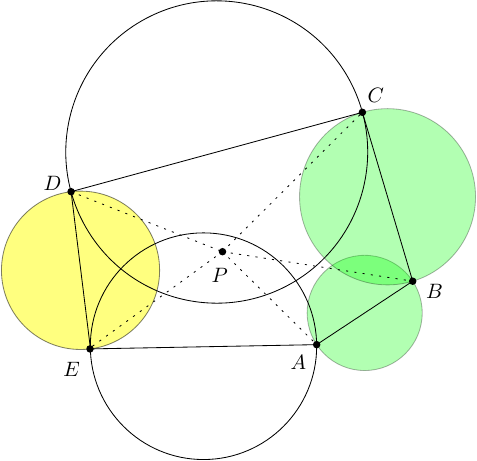}
		\label{fig:approach2}
	}
	\caption{\small{
		(a) Two disjoint disks with diameters $AB$ and $CD$.
		(b) $P$ lies outside the disks with diameters $AB$, $BC$, and $DE$.
	}}
	\label{fig:base-lemmas2}
\end{figure}

\paragraph*{Notation:}
Given three different points $p$, $q$, and $r$ in the plane,
let $\ell(p,q)$ denote the straight line containing both $p$ and $q$,
$pq\subset \ell(p,q)$ the segment with endpoints $p$ and $q$, 
$h(p,q)$ the halfline 
emanating from $p$ and containing $q$,
$\Delta pqr$ the triangle with vertex set $\{p,q,r\}$,
and $\angle pqr$ the angle not bigger than $\pi$ with vertex
$q$ and sides $h(q,p)$ and $h(q,r)$.
For a line $\ell$, let $dist(p,\ell)$ denote the distance
from $p$ to $\ell$.
Given a segment $s$, let $\seg{s}$ denote the length of $s$, $\ell(s)$ the line
that contains $s$, and 
$D_s$ the disk that has diameter $\seg{s}$ and center the midpoint of $s$.
We say that a (convex) quadrilateral is
{\em tangential} if each of its sides is tangent to the same given
disk contained in the quadrilateral. 
Every time we define a polygon
by enumerating its vertices, the vertices are given in counter-clockwise order. We will also
refer to a polygon by giving a sequence of its vertices in counter-clockwise order.



\section{Preliminaries}\label{sec:pre}

Let $s_0,s_1,\ldots,s_{n-1}$ denote in counter-clockwise order the sides of a convex polygon $P_n$. Let $D_0,D_1,\ldots,D_{n-1}$ be the side disks of $P_n$ at $s_0,s_1,\ldots,s_{n-1}$, respectively, and $G=\langle V,E\rangle$ the intersection graph of $D_0,D_1,\ldots,D_{n-1}$. Note that $\{D_i,D_{i+1}\}\in E$ for every $i\in \{0,1,\ldots,n-1\}$, where subindices are taken modulo $n$. Then, $G$ is Hamiltonian, with the cycle $c=\langle D_0,D_1,\ldots,D_{n-1},D_0\rangle$. 

Any Hamiltonian graph $G=\langle V,E\rangle$ with a Hamiltonian cycle $\langle v_0,v_1,\ldots,v_{n-1},v_0\rangle$ can be embedded in the plane as follows: the vertices $V=\{v_0,v_1,\ldots,v_{n-1}\}$ are different points of the unit circle so that the edges of the cycle are the circular arcs between consecutive points, and any other edge $\{v_i,v_j\}\in E$ is the straight chord of the circle, denoted $\mathtt{c}_{i,j}$, that connects the points representing $v_i$ and $v_j$, respectively. We call such an embedding as the \DEF{circular embedding} of $G$.  The chords induce the intersection graph $G_{\mathtt{c}}=\langle V_{\mathtt{c}}, E_{\mathtt{c}}\rangle$, where $V_{\mathtt{c}}$ is the set of chords, and $\{\mathtt{c}_{i,j},\mathtt{c}_{k,\ell}\}\in E_{\mathtt{c}}$ if and only if the chords $\mathtt{c}_{i,j}$ and $\mathtt{c}_{k,\ell}$ have an interior point in common. Observe that subindices $i$, $j$, $k$, and $ \ell$ must be different. See Figure~\ref{fig:examples} for examples.

\begin{figure}[t]
	\centering
	\subfloat[]{
		\includegraphics[scale=0.8,page=23]{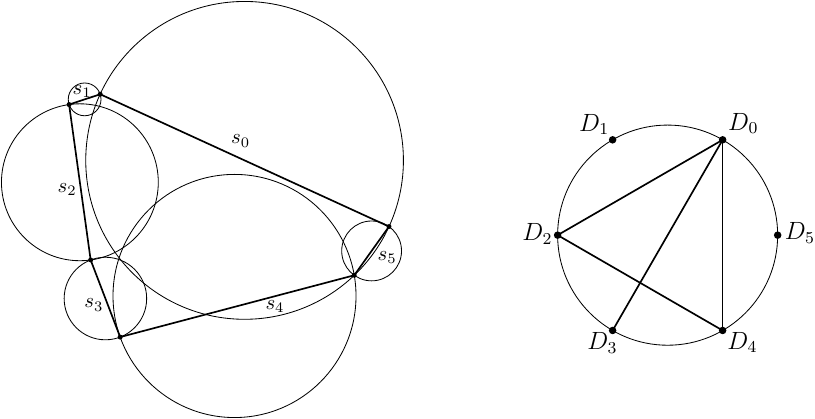}
		\label{fig:example1}
	}\\
	\subfloat[]{
		\includegraphics[scale=0.78,page=24]{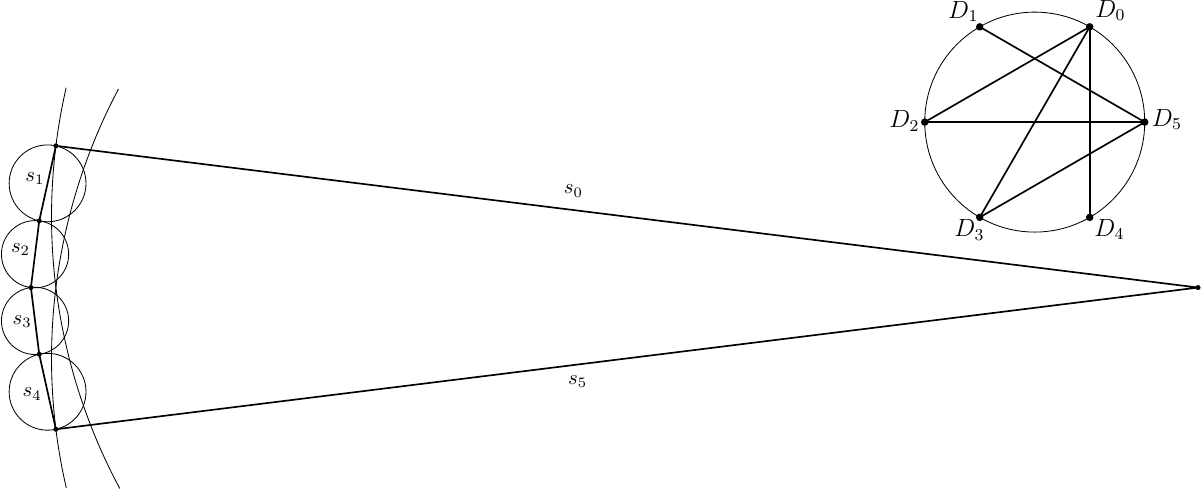}
		\label{fig:example2}
	}
	\caption{\small{
		Two examples of an hexagon withs sides $s_0,s_1,\ldots,s_5$,
		together with the circular embedding of the intersection graph of the side disks.
	}}
	\label{fig:examples}
\end{figure}

Kuratowski and Wagner theorems are well known characterizations of planar graphs, but they are oriented to general graphs. In our particular case, we consider the condition that the intersection graph of the side disks of a convex polyogn has a Hamiltonian cycle, and use the following characterization instead, which is due to Hundack and Stamm-Wilbrandt~\cite{hundack1994planar}.

\begin{theorem}[\cite{hundack1994planar}]\label{theo:hamiltonian-planar}
Let $G=\langle V,E\rangle$ be a Hamiltonian graph, and $G_{\mathtt{c}}=\langle V_{\mathtt{c}}, E_{\mathtt{c}}\rangle$ the intersection graph of the chords in the circular embedding of $G$. Then, $G$ is planar if and only if $G_{\mathtt{c}}$ is bipartite.
\end{theorem}

Any convex $n$-gon is the intersection of $n$ halfplanes, where the boundary of each halfplane contains a side of the $n$-gon. In general, the intersection of $n$ halfplanes is not always a convex polygon: it can be a convex unbounded set whose boundary is a connected polyline with the first and last sides being halflines instead of segments. We say that such a convex set is an \DEF{unbounded} convex $n$-gon, and if $s_0,s_1,\ldots,s_{n-1}$ denote the sides in counter-clockwise order, then $s_0$ and $s_{n-1}$ are the first and last sides, that is, $s_0$ and $s_{n-1}$ are halflines and $s_1,\ldots,s_{n-2}$ are segments. We consider in the case of an unbouded convex $n$-gon that $s_0$ and $s_{n-1}$ \DEF{are not} consecutive sides, and the side disks at them are halfplanes as degenerated disks. In our proof we use both convex $n$-gons and unbounded convex $n$-gons. Given two sides of an (unbounded) convex polygon, the \DEF{bisector} is the line that contains the points of the polygon that are equidistant from the two sides.

To prove that $G_{\mathtt{c}}$ is bipartite, we will show that it does not have cycles of odd length, and the main results that we obtain in this direction are the following ones:

\begin{lemma}[{\bf 1-Chord}]\label{lem:atmost1}
Let $P_n$ be an (unbounded) convex $n$-gon, $n\ge 5$, with sides denoted $s_0,s_1,\ldots,$ $s_{n-1}$ in counter-clockwise order. Let $D_0,D_1,\ldots,D_{n-1}$ be the side disks of $P_n$ at $s_0,s_1,\ldots,$ $s_{n-1}$, respectively. Then, there exists a side $s_i$ such that the disk $D_i$ intersects at most one disk among the disks $D_{i+2},D_{i+3},\dots,D_{i-3},D_{i-2}$ not neighbouring $D_i$, where subindices are taken modulo $n$. That is, there is at most one chord with endpoint the point representing $D_i$ in the circular embedding of the intersection graph of  $D_0,D_1,\ldots,D_{n-1}$.
\end{lemma}

\begin{lemma}[{\bf No-3-Cycles}]\label{lem:no3cycles}
Let $P_n$ be an (unbounded) convex $n$-gon, $n\ge 6$, and $\mathsf{a}$, $\mathsf{b}$, $\mathsf{c}$, $\mathsf{d}$, $\mathsf{e}$, $\mathsf{f}$ six sides appearing in this order counter-clockwise. At least one of the following statements is satisfied:
\begin{itemize}
\item[(a)] $D_\mathsf{a}$ and $D_\mathsf{d}$ do not intersect.
\item[(b)] $D_\mathsf{b}$ and $D_\mathsf{e}$ do not intersect.
\item[(c)] $D_\mathsf{c}$ and $D_\mathsf{f}$ do not intersect.
\end{itemize}
That is, the intersection graph of the chords in the circular embedding of the side disks of $P_n$ does not have 3-length cycles.
\end{lemma}

\begin{theorem}[{\bf Main}]\label{theo:main}
In any convex polygon the intersection graph of the side disks is planar.
\end{theorem}

We prove the above three results in sections~\ref{sec:partI},~\ref{sec:partII}, and~\ref{sec:partIII}, respectively. In each section, before proving the corresponding main result, we also prove several technical lemmas, with arguments mostly from Euclidean geometry. The next basic results will be used:

\begin{figure}[t]
	\centering
	\subfloat[]{
		\includegraphics[scale=0.75,page=1]{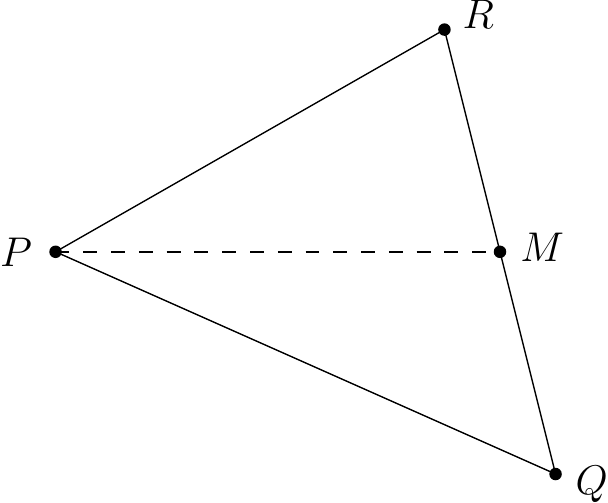}
		\label{fig:median}
	}\hspace{0.2cm}
	\subfloat[]{
		\includegraphics[scale=0.75,page=2]{img.pdf}
		\label{fig:tangential}
	}
	\caption{\small{
		(a) Illustration of Theorem~\ref{lem:median}}.
		(b) Illustration of Lemma~\ref{lem:tangential}.
	}
	\label{fig:base-lemmas}
\end{figure}

\begin{theorem}[Apolonio's Theorem]\label{lem:median}
Let $P$, $Q$, and $R$ be three different points of the plane, and let $M$ denote
the midpoint of the segment $QR$ (see Figure~\ref{fig:median}). Then, the length $\seg{PM}$ satisfies:
\[
	\seg{PM} ~=~ \frac{1}{2} \sqrt{2\left(\seg{PQ}^2 + \seg{PR}^2\right) - \seg{QR}^2}.
\]
\end{theorem}

A known fact that we will also use is the following one: Given a disk and a point
outside it, the two lines passing through the point and tangent to the disk define
two segments of equal lengths. Each segment connects the point with a point of tangency
between one of the lines and the disk.

\begin{lemma}[Diagonal of a tangential quadrilateral~\cite{hajja2008}]\label{lem:tangential}
Let $P$, $Q$, $R$, and $S$ be the vertices of a tangential quadrilateral, tangent
to the disk $\mathcal{C}$. Let $T_1$, $T_2$, $T_3$, and $T_4$ denote the tangent
points between the sides $PQ$, $QR$, $RS$, and $SP$ and $\mathcal{C}$, respectively.
Let $a=\seg{PT_1}=\seg{PT_4}$, $b=\seg{ST_4}=\seg{ST_3}$, $c=\seg{RT_3}=\seg{RT_2}$, and
$d=\seg{QT_2}=\seg{QT_1}$ (see Figure~\ref{fig:tangential}). Then, the length $\seg{PR}$ satisfies:
\[
	\seg{PR} ~=~ \sqrt{ \frac{a+c}{b+d} \cdot \bigl( (a+c)(b+d) + 4bd \bigr)}.
\]
\end{lemma}

\section{Part I: Proof of 1-Chord lemma (Lemma~\ref{lem:atmost1})}\label{sec:partI}

\begin{figure}[t]
	\centering
	\includegraphics[scale=0.8,page=3]{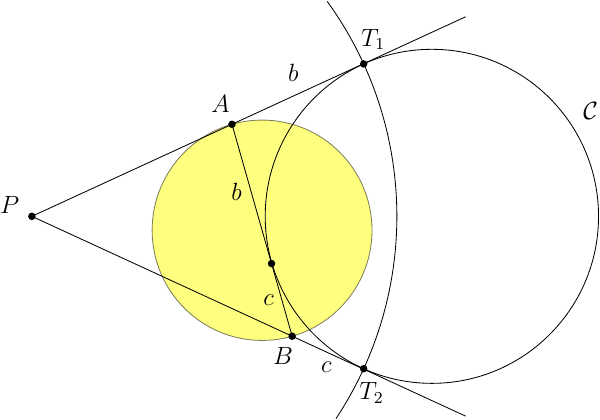}
	\caption{\small{Illustration of Lemma~\ref{lem:1}}.}
	\label{fig:lem1}
\end{figure}

\begin{lemma}\label{lem:1}
Let $\mathcal{C}$ be a disk and let $P$ be a point not contained in $\mathcal{C}$.
Let $T_1$ and $T_2$ be the points of the boundary of $\mathcal{C}$ such that the lines
$\ell(P,T_1)$ and $\ell(P,T_2)$ are tangents to $\mathcal{C}$. Let $A$ be a point in the segment
$PT_1$ and $B$ a point in the segment $PT_2$ such that the segment $AB$ is tangent to $\mathcal{C}$
(see Figure~\ref{fig:lem1}).
Then, the disk $D_{AB}$ is contained in the disk with center $P$
and radius $\seg{PT_1}=\seg{PT_2}$.
\end{lemma}

\begin{proof}
Let $a=\seg{PT_1}=\seg{PT_2}$, $b=\seg{AT_1}$, and $c=\seg{BT_2}$. Let $M$ denote the midpoint
of the segment $AB$, and note that $\seg{PA}=a-b$, $\seg{PB}=a-c$, and $\seg{AB}=b+c$. To prove
the result, it suffices to prove that
\[
	\seg{PM} + \seg{MA} ~\le~ a.
\]
Note that $\seg{PM}=(1/2)\sqrt{2((a-b)^2 + (a-c)^2) - (b+c)^2}$ by Theorem~\ref{lem:median},
and that $\seg{MA}=(b+c)/2$. Since $b\le a$ and $c\le a$, which implies $(b+c)/2\le a$,
verifying the above inequation is equivalent to proving that
\[
	4\cdot \seg{PM}^2 ~=~ 2((a-b)^2 + (a-c)^2) - (b+c)^2 ~\le~ \left(2a - (b+c)\right)^2.
\]
This last equation holds since the following inequalities are equivalent
\begin{eqnarray*}
	2((a-b)^2 + (a-c)^2) - (b+c)^2 & \le & \left(2a - (b+c)\right)^2\\
	4a^2 - 4ab - 4ac - 2bc + b^2 + c^2 & \le & 4a^2 + b^2 + c^2 - 4ab - 4ac + 2bc \\
									0 & \le & 4bc.
\end{eqnarray*}
The result thus follows.
\end{proof}

\begin{figure}[t]
	\centering
	\includegraphics[scale=0.8,page=4]{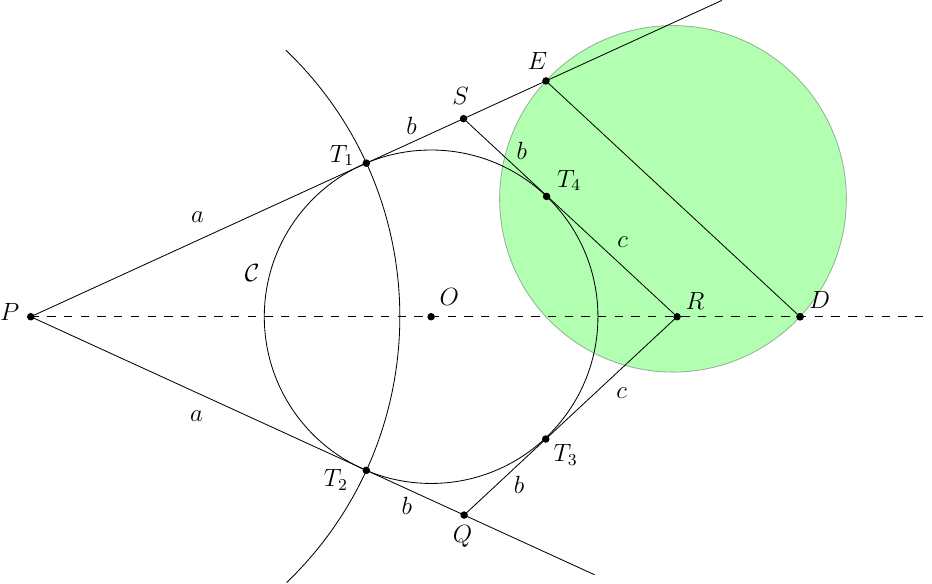}
	\caption{\small{Illustration of Lemma~\ref{lem:2}}.}
	\label{fig:lem2}
\end{figure}

\begin{lemma}\label{lem:2}
Let $\mathcal{C}$ be a disk centered at the point $O$, and let $P$ be a point not contained in $\mathcal{C}$.
Let $T_1$ and $T_2$ be the points of the boundary of $\mathcal{C}$ such that the lines
$\ell(P,T_1)$ and $\ell(P,T_2)$ are tangents to $\mathcal{C}$.
Let $E$ be a point in the halfline $h(P,T_1)\setminus PT_1$ and
$D$ a point in the halfline $h(P,O)$
such that: $\ell(E,D)$ does not
intersect the interior of $\mathcal{C}$, and $\angle EDP\le \pi/2$ (see Figure~\ref{fig:lem2}).
Then, the disk $D_{DE}$ does not
intersect the disk with center $P$ and radius $\seg{PT_1}=\seg{PT_2}$.
\end{lemma}

\begin{proof}
Let $\mathcal{C}_P$ be the disk with center $P$ and radius $\seg{PT_1}=\seg{PT_2}$.
Let $S\in h(P,T_1)\setminus PT_1$ and $R\in h(P,O)\setminus PO$ be the points such that
the line $\ell(S,R)$ is parallel to $\ell(E,D)$ and tangent to $\mathcal{C}$
at the point $T_4$.
Let $Q$ denote the reflection point of $S$ about the line $\ell(P,O)$,
and note that the quadrilateral with vertices $P$, $Q$, $R$, and $S$ is a tangential
quadrilateral, tangent to $\mathcal{C}$.
Let $T_3$ be the point of tangency between the segment $QR$ and $\mathcal{C}$, and
$a=\seg{PT_1}=\seg{PT_2}$, $b=\seg{ST_1}=\seg{ST_4}=\seg{QT_2}=\seg{QT_3}$, and
$c=\seg{RT_3}=\seg{RT_4}$. Then, by Lemma~\ref{lem:tangential} used with $d=b$,
we have that
\[
	\seg{PR} ~=~ \sqrt{ (a+c)(a+c + 2b)}.
\]
Let $M$ denote the midpoint of the segment $SR$, which satisfies that $\seg{MS}=(b+c)/2$.
We claim that
\[
	\seg{PT_1} + \seg{MS} ~=~ a + (b+c)/2 ~<~ \seg{PM}.
\]
Indeed, by Theorem~\ref{lem:median}, we have that
\begin{eqnarray*}
	\seg{PM} & = & \frac{1}{2} \sqrt{2\left(\seg{PS}^2 + \seg{PR}^2\right) - \seg{SR}^2}\\
	         & = & \frac{1}{2} \sqrt{2\left((a+b)^2 + (a+c)(a+c + 2b)\right) - (b+c)^2},
\end{eqnarray*}
and the inequalities
\begin{eqnarray*}
	2a + b + c & < & 2\cdot\seg{PM}\\
	(2a + b + c)^2 & < & 2\left((a+b)^2 + (a+c)(a+c + 2b)\right) - (b+c)^2\\
	4a^2 + b^2 + c^2 + 4ab + 4ac + 2bc & < & 2(2a^2 + b^2 + c^2 + 4ab + 2ac + 2bc) - (b^2+c^2+2bc)\\
	4a^2 + b^2 + c^2 + 4ab + 4ac + 2bc & < & 4a^2 + b^2 + c^2 + 8ab + 4ac + 2bc\\
									0  & < & 4ab,
\end{eqnarray*}
are all equivalent and hold given that $a,b>0$, which imply the claim.
Let $M'$ denote the midpoint of the segment $ED$. Since triangles
$\Delta PRS$ and $\Delta PDE$
are similar, we have that $\seg{PM'}=\lambda \cdot \seg{PM}$ and $\seg{M'E}=\lambda\cdot \seg{MS}$,
where $\lambda = \seg{PE}/\seg{PS}=\seg{PD}/\seg{PR}=\seg{ED}/\seg{SR}\ge 1$ is the similarity ratio
between these triangles. Then, since $\seg{PM} - \seg{MS} > \seg{PT_1} > 0$, we have that
\[
	\seg{PT_1} ~<~ \seg{PM} - \seg{MS} ~\le~ \lambda(\seg{PM} - \seg{MS}) ~=~ \seg{PM'} - \seg{M'E}.
\]
This immediately implies that the disk $D_{DE}$ does not intersect the disk $\mathcal{C}_P$.
\end{proof}

\begin{lemma}\label{lem:3}
Let $\mathcal{C}$ be a disk centered at the point $O$, and let $P$ be a point not contained in $\mathcal{C}$.
Let $T_1$ and $T_2$ be the points of the boundary of $\mathcal{C}$ such that the lines
$\ell(P,T_1)$ and $\ell(P,T_2)$ are tangents to $\mathcal{C}$.
Let $E$ be a point in the halfline $h(P,T_1)\setminus PT_1$ and
$D$ a point in halfline $h(P,O)$
such that: $\ell(E,D)$ does not
intersect the interior of $\mathcal{C}$, and $\angle EDP > \pi/2$ (see Figure~\ref{fig:lem4}).
Then, the disk $D_{DE}$ does not
intersect the disk $\mathcal{C}_P$ with center $P$ and radius $\seg{PT_1}=\seg{PT_2}$.
\end{lemma}

\begin{figure}[t]
	\centering
	\includegraphics[scale=0.7,page=22]{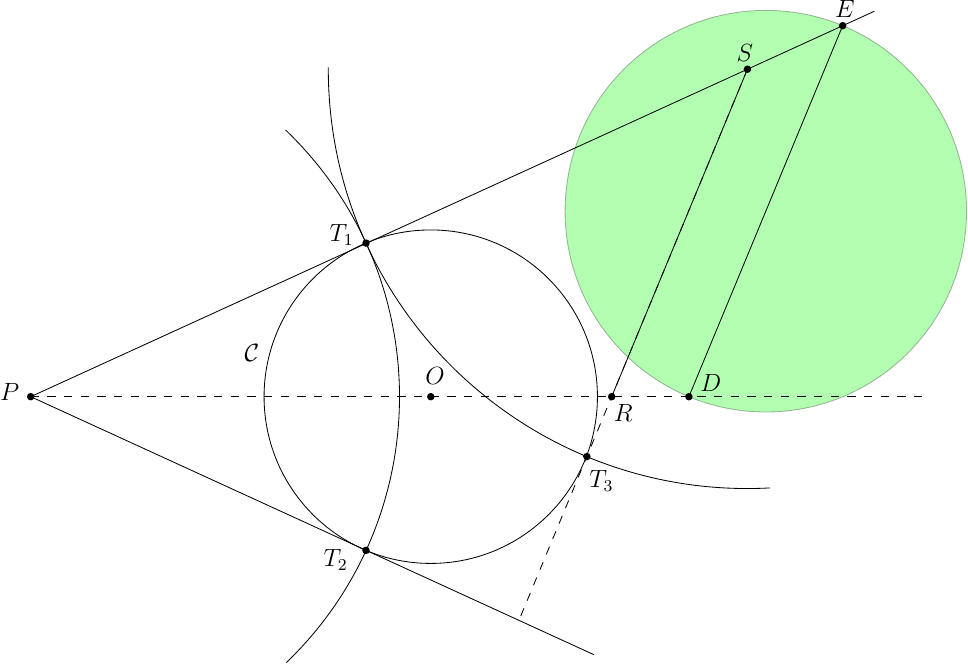}
	\caption{\small{Illustration of Lemma~\ref{lem:3}}.}
	\label{fig:lem3}
\end{figure}

\begin{proof}
Let $S\in h(P,T_1)\setminus PT_1$ and $R\in h(P,O)\setminus PO$ be the points such that
the line $\ell(S,R)$ is parallel to $\ell(E,D)$ and tangent to the disk $\mathcal{C}$
at the point $T_3$. Note that $T_3$ belongs to the wedge bounded by $h(P,O)$ and $h(P,T_2)$.
Let $\mathcal{C}_S$ be the disk with center $S$ and radius $\seg{PT_1}=\seg{PT_3}$ (see Figure~\ref{fig:lem3}).
Since $\mathcal{C}_P$ and $\mathcal{C}_S$ have disjoint interiors and $D_{SR}\subset \mathcal{C}_S$,
then $\mathcal{C}_P$ and $D_{SR}$ are disjoint. Similar as in the last arguments of the proof of Lemma~\ref{lem:2},
$D_{DE}$ does not intersect $\mathcal{C}_P$.
\end{proof}

\begin{lemma}\label{lem:4}
Let $\mathcal{C}$ be a disk centered at the point $O$, and let $P$ be a point not contained in $\mathcal{C}$.
Let $T_1$ and $T_2$ be the points of the boundary of $\mathcal{C}$ such that the lines
$\ell(P,T_1)$ and $\ell(P,T_2)$ are tangents to $\mathcal{C}$.
Let $E$ be a point in the halfline $h(P,T_1)\setminus PT_1$ and
$D\neq E$ a point in the interior of the convex wedge bounded by $h(P,T_1)$ and $h(P,O)$
such that: $h(E,D)$ does not intersect with $h(P,O)$, and $\ell(E,D)$ does not
intersect the interior of $\mathcal{C}$ (see Figure~\ref{fig:lem4}).
Then, the disk $D_{DE}$ does not
intersect the disk with center $P$ and radius $\seg{PT_1}=\seg{PT_2}$.
\end{lemma}

\begin{figure}[t]
	\centering
	\includegraphics[scale=0.75,page=5]{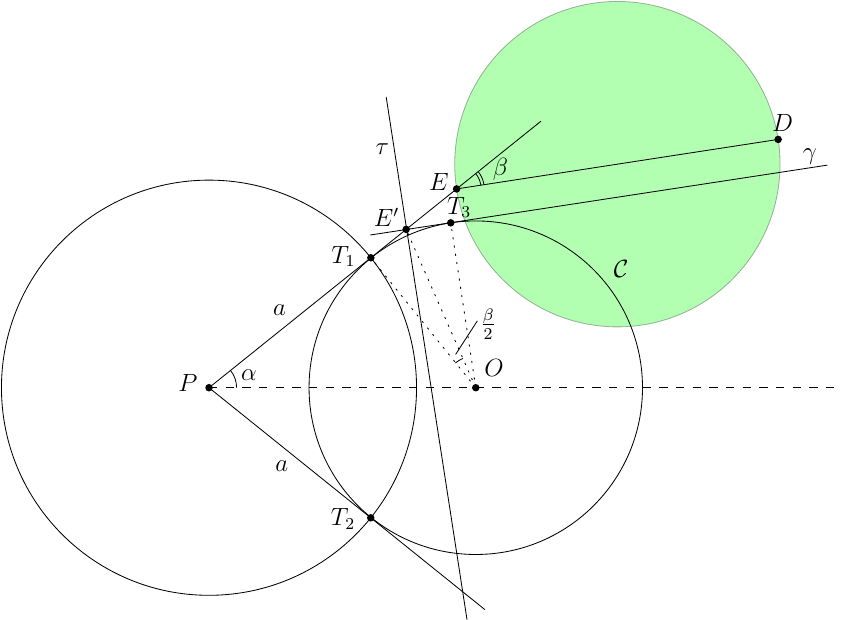}
	\caption{\small{Illustration of Lemma~\ref{lem:4}}.}
	\label{fig:lem4}
\end{figure}

\begin{proof}
Let $\alpha\in(0,\pi/2)$ denote the angle formed by $h(P,E)$ and $h(P,O)$,
and $\beta\in(0,\alpha]$ the angle formed by $h(P,E)$ and $\ell(E,D)$ (see Figure~\ref{fig:lem4}).
Let $E'$ be the point in $h(P,T_1)\setminus PT_1$ such that the line different from $\ell(P,E')$ containing
$E'$ and tangent to $\mathcal{C}$, denoted $\gamma$, is parallel to the line $\ell(E,D)$. Let
$\tau$ be the line perpendicular to $\gamma$ that contains $E'$,
and $T_3$ denote the point of tangency between $\gamma$ and $\mathcal{C}$.
Let $a=\seg{PT_1}=\seg{PT_2}$, $b=\seg{E'T_1}=\seg{E'T_3}$, and $r$ denote the radius of $\mathcal{C}$.
Since the lines $\gamma$ and $\tau$ are perpendicular, the angle formed by the lines $\ell(P,E')$ and
$\tau$ is equal to $\pi/2 - \beta$. Then, the distance $dist(P,\tau)$ from the point $P$ to the line $\tau$ satisfies
\[
	dist(P,\tau) ~=~ \seg{PE'}\cdot\sin(\pi/2 - \beta) ~=~ (a+b)\cdot\cos\beta.
\]
Note that $\angle T_1E'T_3=\pi-\beta$. Then, since the line $\ell(O,E')$ bisects the angle $\angle T_1E'T_3$,
we have that $\angle T_1E'O=\pi/2-\beta/2$, which implies that 
\[
	b= r\cdot \cot(\angle T_1E'O)=r\cdot \tan(\beta/2)
\]
because the segment $OT_1$ satisfying $\seg{OT_1}=r$ is perpendicular to the line $\ell(P,E')$.
On the other hand, note that $a=r\cdot\cot\alpha$.
Putting the above observations together, the next inequalities 
\begin{eqnarray*}
	\seg{PT_1} & \le & dist(P,\tau)\\
	a & \le & (a+b)\cdot\cos \beta\\
	r\cdot \cot\alpha & \le & (r\cdot \cot\alpha + r\cdot \tan(\beta/2))\cdot \cos\beta\\
	\cot \alpha & \le & \frac{\cos\beta\cdot\tan(\beta/2)}{1-\cos\beta} ~=~
				\frac{\cos\beta\cdot \frac{\sin(\beta/2)}{\cos(\beta/2)}}{2\sin^2(\beta/2)} ~=~
				\frac{\cos\beta}{2\sin(\beta/2)\cos(\beta/2)}\\
	\frac{\cos\alpha}{\sin\alpha} & \le & \frac{\cos\beta}{\sin\beta}\\
	0 & \le & \sin(\alpha - \beta)
\end{eqnarray*}
are all equivalent and hold given that $\beta>0$ and $0\le \alpha-\beta < \alpha < \pi/2$.
Since by construction the line $\tau$ either does not intersect the disk $D_{DE}$
or is tangent to $D_{DE}$ at the point $E'\neq T_1$, we can guarantee that
the disk $D_{DE}$ does not intersect the disk with center $P$ and radius $\seg{PT_1}=\seg{PT_2}$.
The lemma thus follows.
\end{proof}

\begin{lemma}\label{lem:5}
Let $\mathcal{C}$ be a disk centered at the point $O$, and let $P$ be a point not contained in $\mathcal{C}$.
Let $T_1$ and $T_2$ be the points of the boundary of $\mathcal{C}$ such that the lines
$\ell(P,T_1)$ and $\ell(P,T_2)$ are tangents to $\mathcal{C}$.
Let $D$ be a point in the halfline $h(P,O)\setminus PO$ and
$E$ a point in the interior of the convex wedge bounded by $h(P,T_1)$ and $h(P,O)$
such that $h(D,E)$ does not intersect with $h(P,T_1)$ (see Figure~\ref{fig:lem5}).
Then, the disk $D_{DE}$ does not
intersect the disk $\mathcal{C}_P$ with center $P$ and radius $\seg{PT_1}=\seg{PT_2}$.
\end{lemma}

\begin{proof}
Let $\tau$ be the line through $D$ that is perpendicular to $DE$.
Let $D'$ be the orthogonal projection of $D$ into $h(P,T_1)$, that is,
lines $\ell(D,D')$ and $\ell(P,T_1)$ are perpendicular at $D'$. By the definition of $E$, 
the distante from $P$ to $\tau$ is at least $\seg{PD'}$ and at most $\seg{PD}$.
Since $\seg{PT_1}<\seg{PD}$, then the disks $D_{DE}$ and $\mathcal{C}_P$ do not intersect.
\end{proof}

\begin{figure}[t]
	\centering
	\includegraphics[scale=0.65,page=25]{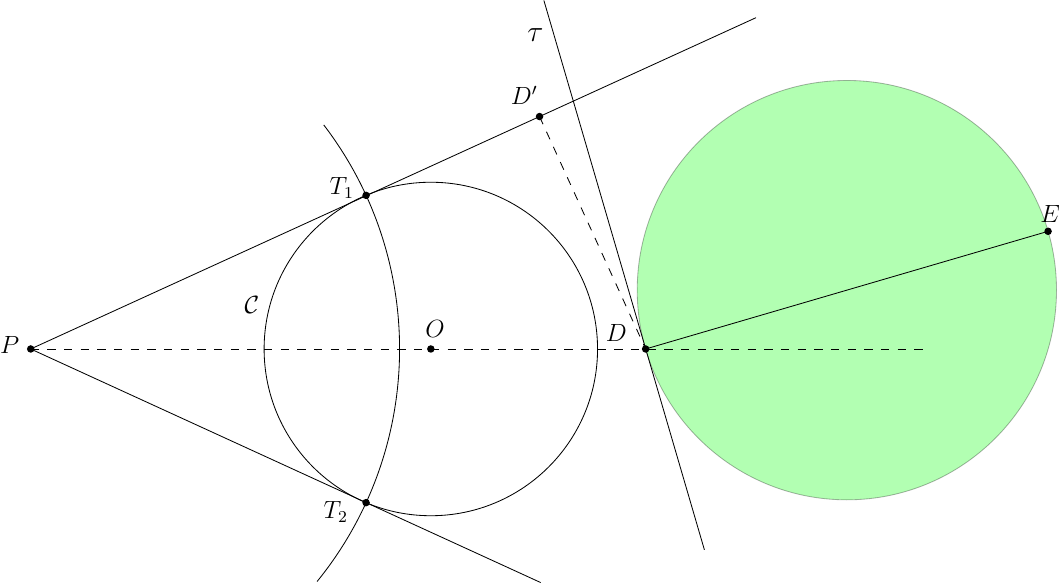}
	\caption{\small{Illustration of Lemma~\ref{lem:5}}.}
	\label{fig:lem5}
\end{figure}

\begin{figure}[t]
	\centering
	\subfloat[]{
		\includegraphics[scale=1.1,trim=6.8cm 17.5cm 8.8cm 6.3cm,clip]{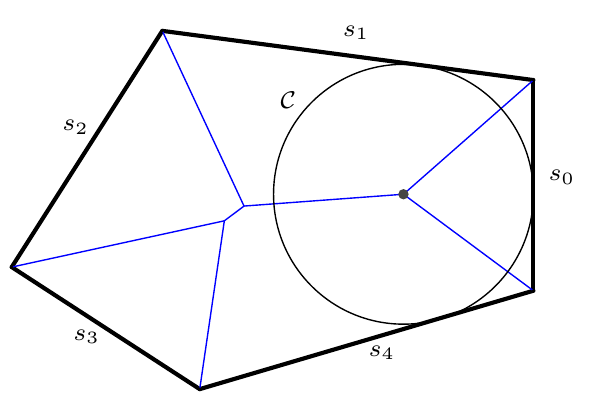}
		\label{fig:medial-axis}
	}
	\subfloat[]{
		\includegraphics[scale=1.1,trim=8cm 17cm 7cm 6.2cm,clip]{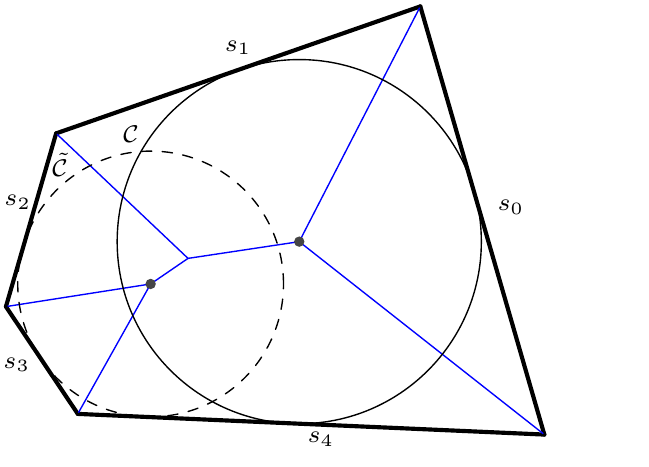}
		\label{fig:case2}
	}\\
	\subfloat[]{
		\includegraphics[scale=1.1,trim=7.8cm 18cm 6.2cm 5.5cm,clip]{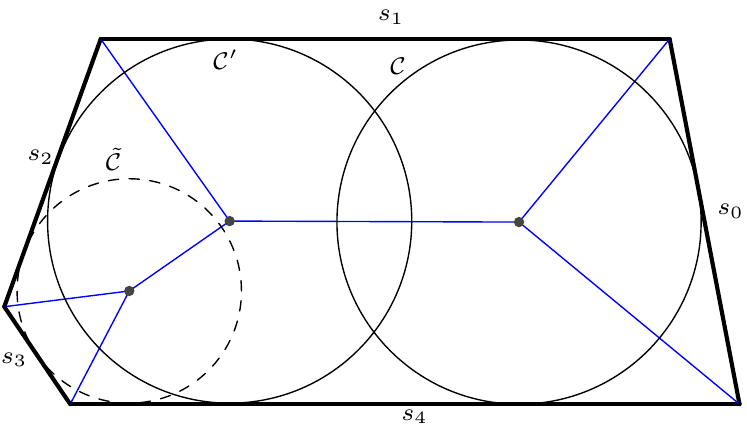}
		\label{fig:case3}
	}
	\captionsetup{justification=justified,singlelinecheck=true}
	\caption{\small{
		Proof of Lemma~\ref{lem:disk}.
		(a) The medial axis of a convex pentagon with sides $s_0$, $s_1$, $s_2$, $s_3$, $s_4$, and
		a disk $\mathcal{C}$ tangent to the sides $s_4$, $s_0$, and $s_1$ and centered at a vertex 
		of the medial axis.
		(b) If $\ell(s_4)$ and $\ell(s_1)$ are not parallel and their intersection point and
		the interior of $\mathcal{C}$ are at the same halfplane bounded by $\ell(s_0)$, and $\mathcal{C}$ 
		is not tangent to both $s_2$ and $s_3$, then there exists a disk $\tilde{\mathcal{C}}$ with
		smaller radius tangent to three consecutive sides.
		(c) If $\ell(s_4)$ and $\ell(s_1)$ are parallel and $\mathcal{C}'$ is not tangent to $s_3$,
		then there exists a disk $\tilde{\mathcal{C}}$ with
		smaller radius tangent to three consecutive sides.
	}}
	\label{fig:maxis}
\end{figure}

\begin{lemma}\label{lem:disk}
Any (unbounded) convex $n$-gon, $n\ge 5$, contains a disk $\mathcal{C}$ tangent to three consecutive sides, 
such that: the lines containing the first 
and third sides, respectively, are not parallel and further their intersection point and
the interior of $\mathcal{C}$ belong to different halfplanes bounded by the line containing the second side.
\end{lemma}

\begin{proof}
Let $P_n$ be a convex $n$-gon with sides denoted $s_0,s_1,\ldots,s_{n-1}$ counter-clockwise. 
In the following, every disk will be considered
to be contained in $P_n$, and for every side $s$, let $\ell(s)$ denote
the line containing $s$. There exist disks tangent to three consecutive sides and
centered at a vertex of the medial axis of $P_n$~\cite{preparata1977}.
The medial axis of a simple polygon is the locus 
of the points of the polygon that have more than one closest point in the boundary.
If the polygon is convex, the medial axis is a tree made of line segments, each contained in the bisector
of two sides (see Figure~\ref{fig:medial-axis}).
Then, let $\mathcal{C}$ be a disk of minimum radius among those disks, tangent
to the sides $s_{n-1}$, $s_0$, and $s_1$ w.l.o.g.\ 
If $\ell(s_{n-1})$ and $\ell(s_{1})$ 
are not parallel and their intersection point and the interior of $\mathcal{C}$ 
are at different halfplanes bounded by $\ell(s_0)$ (see Figure~\ref{fig:medial-axis}), 
then the lemma is proved.
Otherwise, if $\ell(s_{n-1})$ and $\ell(s_{1})$ 
are not parallel and their intersection point and the interior of $\mathcal{C}$ 
are at the same halfplane bounded by $\ell(s_0)$  (see Figure~\ref{fig:case2}),
then by the minimality of $\mathcal{C}$ 
every side among $s_{2},s_{3},\ldots,s_{n-2}$ must be tangent to $\mathcal{C}$,
which implies that every triple of consecutive sides among $s_{1},s_{2},\ldots,s_{n-1}$
together with $\mathcal{C}$ satisfy the conditions of the lemma.
Finally, if $\ell(s_{n-1})$ and $\ell(s_{1})$ are parallel (see Figure~\ref{fig:case3}), 
then by the minimality of $\mathcal{C}$
the disk $\mathcal{C'}$ with radius equal to that of $\mathcal{C}$ and tangent
to the sides $s_1$ and $s_{n-1}$, and to at least one side $s_i$ for some $i\in[2\dots n-2]$, must be tangent
to all the sides $s_{2},s_{3},\ldots,s_{n-2}$. Since $n\ge 5$, 
every triple of consecutive sides among $s_{1},s_{2},\ldots,s_{n-1}$ and the disk $\mathcal{C'}$ prove the lemma.
\end{proof}

\begin{figure}[t]
	\centering
	\includegraphics[scale=0.75,page=26]{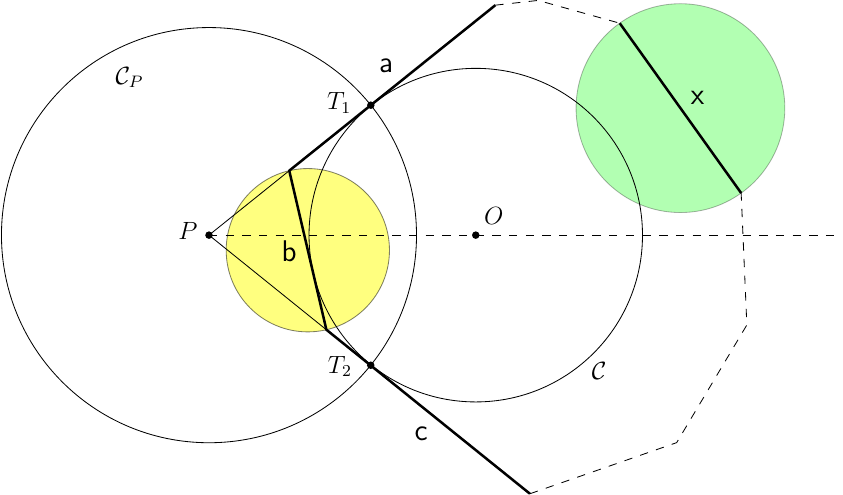}
	\caption{\small{Proof of Lemma~\ref{lem:abcx}}.}
	\label{fig:abcx}
\end{figure}

\begin{lemma}\label{lem:abcx}
Let $P_n$ be an (unbounded) convex $n$-gon, and $\mathsf{a}$, $\mathsf{b}$, and $\mathsf{c}$ three consecutive sides of $P_n$ such that: the lines $\ell(\mathsf{a})$ and $\ell(\mathsf{c})$ intersect at point $P$, the line $\ell(\mathsf{b})$ separates the interior of $P_n$ and $P$, and there exists a disk $\mathcal{C}$ with center $O$, contained in $P_n$, and tangent to $\mathsf{a}$, $\mathsf{b}$, and $\mathsf{c}$. Then, for any side $\mathsf{x}\notin\{\mathsf{a},\mathsf{b},\mathsf{c}\}$ of $P_n$ such that the bisector $\ell(P,O)$ of $\mathsf{a}$ and $\mathsf{c}$ does not intersect the interior of $\mathsf{x}$, we have that $D_\mathsf{b}$ and $D_\mathsf{x}$ do not intersect (see Figure~\ref{fig:abcx}).
\end{lemma}

\begin{proof}
Let $T_1$ and $T_2$ denote the points of tangency between $\mathcal{C}$ and the sides $\mathsf{a}$ and $\mathsf{c}$, respectively, and $\mathcal{C}_P$ the disk with center $P$ and radius $\seg{PT_1}=\seg{PT_2}$. Assume w.l.o.g.\ that $\mathsf{x}$ is contained in the convex wedge bounded by $h(P,T_1)$ and $h(P,O)$. By Lemma~\ref{lem:1}, we have that $D_\mathsf{b}\subset \mathcal{C}_P$. Furthermore, according to the relative position of $\mathsf{x}$ with respect to $h(P,T_1)$, $h(P,O)$, and $\mathcal{C}$, we can use Lemma~\ref{lem:2}, Lemma~\ref{lem:3}, Lemma~\ref{lem:4}, or Lemma~\ref{lem:5} by considering $\mathsf{x}\subseteq DE$ in every of them, to obtain that $\mathcal{C}_P \cap  D_\mathsf{x}$ is empty. Hence, we have that $D_\mathsf{b}$ and $D_\mathsf{x}$ do not intersect.
\end{proof}

\begin{proof}[Proof of 1-Chord lemma (Lemma~\ref{lem:atmost1})]
Using Lemma~\ref{lem:disk}, we can ensure that $P_n$ contains a disk $\mathcal{C}$ tangent to three consecutive sides, 
say the sides $s_{i-1}$, $s_i$, and $s_{i+1}$ for some $i\in\{0,1,\ldots,n-1\}$, such that: the lines $\ell(s_{i-1})$ and $\ell(s_{i+1})$ are not parallel, and their intersection point and
the interior of $\mathcal{C}$ belong to different halfplanes bounded by the line $\ell(s_i)$. The bisector of $s_{i-1}$ and $s_{i+1}$ will cross the interior of at most one side $s_j$ of the set $S=\{s_{0},s_{1},\ldots,s_{n-1}\}\setminus\{s_{i-1},s_{i},s_{i+1}\}$. For any other side $s_k\in S\setminus \{s_j\}$ we have that $D_i$ and $D_k$ do not intersect, by Lemma~\ref{lem:abcx}. The lemma thus follows.
\end{proof}

\section{Part II: Proof of No-3-Cycles lemma (Lemma~\ref{lem:no3cycles})}\label{sec:partII}

\begin{lemma}\label{lem:quad}
Let $ABCD$ be a convex quadrilateral with vertices $A$, $B$, $C$, and $D$, so that the lines 
$\ell(B,C)$ and $\ell(A,D)$ intersect at the point $P$, and the line $\ell(A,B)$ separates
$P$ and the interior of $ABCD$ (see Figure~\ref{fig:quad}(left)). The disk $\mathcal{C}$
with center $O$ is contained in $ABCD$ and tangent to the sides $AB$, $BC$, and $DA$, the
line $\ell(A,O)$ intersects the side $BC$, and the line $\ell(B,O)$ intersects the side $DA$.
Then, the disks $D_{AB}$ and $D_{CD}$ do not intersect.
\end{lemma}

\begin{figure}[t]
	\centering
	\includegraphics[scale=0.8,page=15]{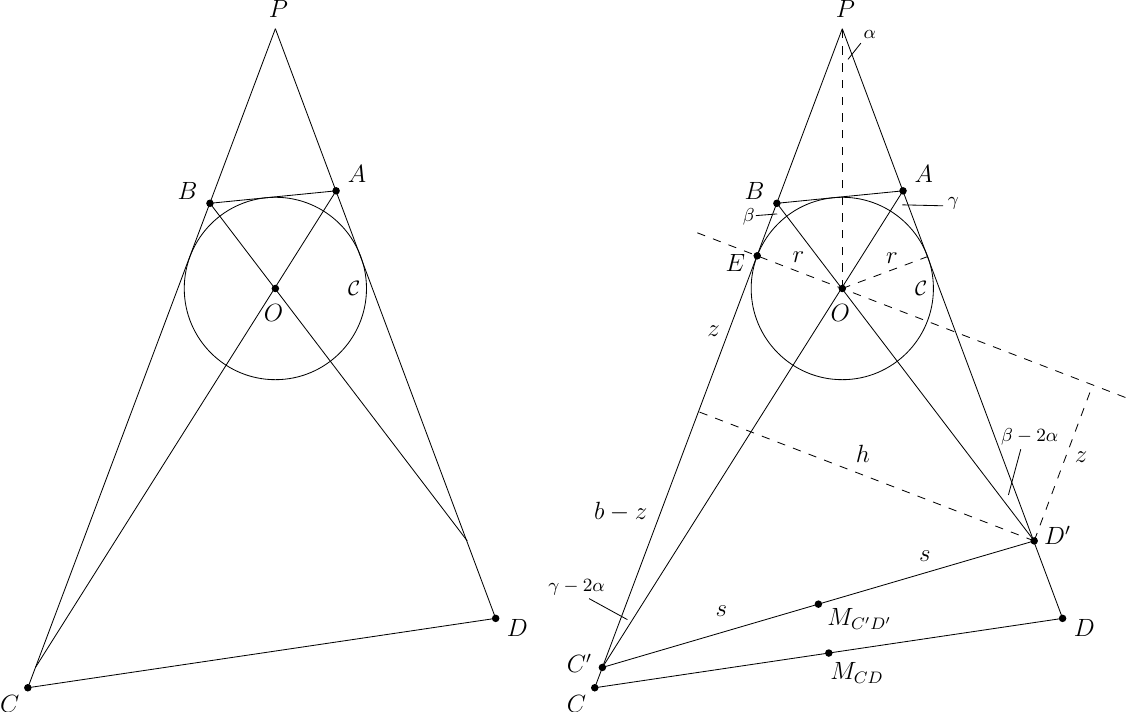}
	\caption{\small{Illustration (left) and proof (right) of Lemma~\ref{lem:quad}}.}
	\label{fig:quad}
\end{figure}

\begin{proof}
(Refer to Figure~\ref{lem:quad}(right) throughout the proof)
Let $E$ be the point of tangency between $\mathcal{C}$ and $BC$, 
and $C'=\ell(A,O) \cap BC$, and $D'=\ell(B,O) \cap DA$.
Let $r$ be the radius of $\mathcal{C}$, $b=\seg{C'E}$, $z=dist(D',\ell(E,O))$, 
$h=dist(D',\ell(P,C))$, $s=\seg{C'D'}/2$, $\alpha=\angle OPA=\angle OPB $,
$\beta = \angle OBA=\angle OBC$, and $\gamma =\angle OAD =\angle OAB$.
Assume w.l.o.g.\ that $\beta \ge \gamma$. Observe that $\gamma > 2\alpha$ since $\ell(A,O)$ intersects $BC$. 
Analogously, $\beta > 2\alpha$ since $\ell(B,O)$ intersects $DA$.
Note also that $\pi + 2\alpha = 2\beta + 2\gamma$.

The disk $\mathcal{C}_P$ with center $P$ and radius $\seg{PE}$ contains
$D_{AB}$, by Lemma~\ref{lem:1}. Then, it suffices to prove that $\mathcal{C}_P$ and $D_{CD}$
do not intersect, which follows by proving that $\mathcal{C}_P$ and $D_{C'D'}$
do not intersect. The reason of this last statement is that any point $Q$ of $D_{CD}$ in the triangle $\Delta PCD$ also belongs to $D_{C'D'}$.
Indeed, $Q\in D_{CD}\cap \Delta PCD$ implies that $\angle CQD\ge \pi/2$ by Thales' theorem, and we also have
$\angle C'QD' > \angle CQD$. Then, $Q$ also belongs to $D_{C'D'}$ by Thales' theorem.
We will
prove in the following that $\ell(E,O)$ separates the interior of $\mathcal{C}_P$
from the whole $D_{C'D'}$.

We need to prove that the radius $s$ of $D_{C'D'}$ is less than the distance
$dist(M_{C'D'},\ell(E,O))=(b+z)/2$, where $M_{C'D'}$ denotes the midpoint of $C'D'$.
That is, we need to show that $(b+z)^2 > (2s)^2$,
where $(2s)^2 = h^2 + (b-z)^2$. This is equivalent to proving that
$4bz>h^2$, with
\begin{align*} 
	b & ~=~  r\cdot \cot(\gamma-2\alpha)~=~ r\cdot \cot(\pi/2+\alpha-\beta-2\alpha)) ~=~ r\cdot\tan(\beta+\alpha), \\ 
	z & ~=~  \seg{D'O}\cdot\cos\beta ~=~ \left(\frac{r}{\sin(\beta-2\alpha)}\right)\cos\beta,
\end{align*}
and 
\[
	h ~=~ r + \seg{D'O}\cdot \sin\beta ~=~  r + \left(\frac{r}{\sin(\beta-2\alpha)}\right)\sin\beta.
\]
This is equivalent to verifying
\[
	4\cdot\frac{\tan(\beta+\alpha)\cos\beta}{\sin(\beta-2\alpha)} ~>~ \left(1 + \frac{\sin\beta}{\sin(\beta-2\alpha)}\right)^2.
\]
Since $2\alpha<\beta$ we have $\sin(\beta-2\alpha) >0$. On the other hand,
given that $\pi + 2\alpha = 2\beta + 2\gamma$ and $\gamma>2\alpha$, 
we have $\beta+\alpha <\beta+\gamma-\alpha=\pi/2$, and then $\cos(\beta+\alpha) > 0$.
Hence, the above inequation is equivalent to
\begin{equation}
	\label{eq1}
	4\cdot \sin(\beta+\alpha)\sin(\beta-2\alpha)\cos\beta ~>~ (\sin(\beta-2\alpha)+\sin\beta)^2\cos(\beta+\alpha).
\end{equation}
Since the sine function is concave in $[0,\pi/2]$, for all $x,y\in[0,\pi/2]$ we have
\[
	\frac{\sin x + \sin y}{2} ~\le~ \sin\left(\frac{x+y}{2}\right),
\]
by Jensen's innequality, and then
\[
	(\sin(\beta-2\alpha)+\sin\beta)^2 ~\le~ \left( 2\cdot\sin\left(\frac{\beta-2\alpha +\beta}{2} \right) \right)^2 
	~=~ 4\cdot\sin^2(\beta-\alpha).
\]
Hence, we have
\[
	 4\cdot \sin^2(\beta-\alpha)\cos(\beta+\alpha) ~\ge~ (\sin(\beta-2\alpha)+\sin\beta)^2\cos(\beta+\alpha).
\]
and then to prove inequation~\eqref{eq1} it suffices to prove 
\begin{equation}
	\label{eq2}
	\sin(\beta+\alpha)\sin(\beta-2\alpha)\cos\beta ~>~ \sin^2(\beta-\alpha)\cos(\beta+\alpha).
\end{equation}
Note that
\begin{eqnarray*}
	&   & \sin(\beta+\alpha)\sin(\beta-2\alpha)\cos\beta \\
	& = & (\sin\beta \cos\alpha + \cos\beta\sin\alpha) (\sin\beta\cos 2\alpha - \cos\beta\sin 2\alpha)\cos\beta \\
	& = & (\sin\beta \cos\alpha + \cos\beta\sin\alpha) (\sin\beta\cos^2\alpha - \sin\beta\sin^2\alpha - 2 \cos\beta\sin \alpha\cos\alpha)\cos\beta \\
	& = & \sin^2\beta\cos\beta\cos^3\alpha - \sin^2\beta\cos\beta\sin^2\alpha\cos\alpha - 2\sin\beta\cos^2\beta\sin\alpha\cos^2\alpha \\
	&   & + \sin\beta\cos^2\beta\sin\alpha\cos^2\alpha - \sin\beta\cos^2\beta\sin^3\alpha - 2\cos^3\beta\sin^2\alpha\cos\alpha,
\end{eqnarray*}
and
\begin{eqnarray*}
	&   & \sin^2(\beta-\alpha)\cos(\beta+\alpha) \\
	& = & (\sin\beta\cos\alpha - \cos\beta\sin\alpha)^2(\cos\beta\cos\alpha - \sin\beta\sin\alpha) \\
	& = & (\sin^2\beta\cos^2\alpha + \cos^2\beta\sin^2\alpha - 2\sin\beta\cos\beta\sin\alpha\cos\alpha)(\cos\beta\cos\alpha - \sin\beta\sin\alpha) \\
	& = & \sin^2\beta\cos\beta\cos^3\alpha + \cos^3\beta\sin^2\alpha\cos\alpha - 2\sin\beta\cos^2\beta\sin\alpha\cos^2\alpha \\
	&   & - \sin^3\beta\sin\alpha\cos^2\alpha - \sin\beta\cos^2\beta\sin^3\alpha + 2\sin^2\beta\cos\beta\sin^2\alpha\cos\alpha.
\end{eqnarray*}
Then, subtracting the above equations, we have
\begin{eqnarray*}
	&   & \sin (\beta+\alpha)\sin(\beta-2\alpha)\cos\beta - \sin^2(\beta-\alpha)\cos(\beta+\alpha)\\
	& = & - 3\cos^3\beta\sin^2\alpha\cos\alpha - 3\sin^2\beta\cos\beta\sin^2\alpha\cos\alpha + \sin^3\beta\sin\alpha\cos^2\alpha \\
	&   & + \sin\beta\cos^2\beta\sin\alpha\cos^2\alpha \\
	& = & \sin\alpha\cos\alpha\left(-3\cos\beta\sin\alpha(\cos^2\beta + \sin^2\beta) + \sin\beta\cos\alpha(\cos^2\beta + \sin^2\beta)\right) \\
	& = & \sin\alpha\cos\alpha\left(-3\cos\beta\sin\alpha + \sin\beta\cos\alpha\right).
\end{eqnarray*}
To prove inequation~\eqref{eq2}, it suffices to show that
\[
	-3\cos\beta\sin\alpha + \sin\beta\cos\alpha ~>~ 0,
\]
that is, $\tan\beta > 3\cdot \tan\alpha$. 
Given that $\pi + 2\alpha = 2\beta + 2\gamma > 8\alpha$, we have
$\alpha<\pi/6$. Furthermore, $\pi + 2\alpha = 2\beta + 2\gamma \le 4\beta$
implies $\beta  \ge \pi/4+\alpha/2$. Then, note that
\begin{eqnarray*}
	\tan\beta  & \ge & \tan\left(\pi/4+\alpha/2\right) ~=~ \frac{\sin(\pi/4+\alpha/2)}{\cos(\pi/4+\alpha/2)} 
			~=~ \frac{\cos(\alpha/2)+\sin(\alpha/2)}{\cos(\alpha/2)-\sin(\alpha/2)} \\
		& = & \frac{\left(\cos(\alpha/2)+\sin(\alpha/2)\right)^2}{\cos^2(\alpha/2)-\sin^2(\alpha/2)} ~=~ \frac{1+\sin\alpha}{\cos\alpha} 
			~>~ 3\cdot \frac{\sin\alpha}{\cos\alpha} ~=~ 3\cdot \tan\alpha,
\end{eqnarray*}
because $\sin\alpha < \sin (\pi/6) = 1/2$ given that $0< \alpha <\pi/6$.
\end{proof}

\begin{figure}[t]
	\centering
	\includegraphics[scale=0.75,page=11]{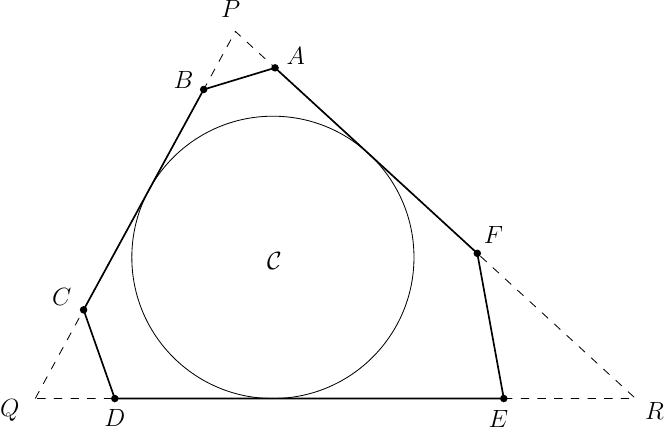}
	\caption{\small{Proof of Lemma~\ref{lem:3pairs}}.}
	\label{fig:3pairs1}
\end{figure}

\begin{lemma}\label{lem:3pairs}
Let $P$, $Q$, $R$, $A$, $B$, $C$, $D$, $E$, and $F$ be points defining the triangle $\Delta PQR$,
and the convex hexagon $ABCDEF$ inscribed in $PQR$ in the following manner: the points $B$ and
$C$ are in $PQ$, the points $D$ and $E$ are in $QR$, the points $F$ and $A$ are in $RP$,
and $ABCDEF$ contains the disk $\mathcal{C}$ incribed to $PQR$ in its interior. Furthermore, $\mathcal{C}$
is tangent to $BC$, $DE$, and $FA$ (see Figure~\ref{fig:3pairs1}).
Then, at least one of the following statements is satisfied:
\begin{itemize}
\item[(a)] $D_{AB}$ and $D_{DE}$ do not intersect.
\item[(b)] $D_{CD}$ and $D_{FA}$ do not intersect.
\item[(c)] $D_{EF}$ and $D_{BC}$ do not intersect.
\end{itemize}  
\end{lemma}

\begin{figure}[t]
	\centering
	\includegraphics[scale=0.75,page=12]{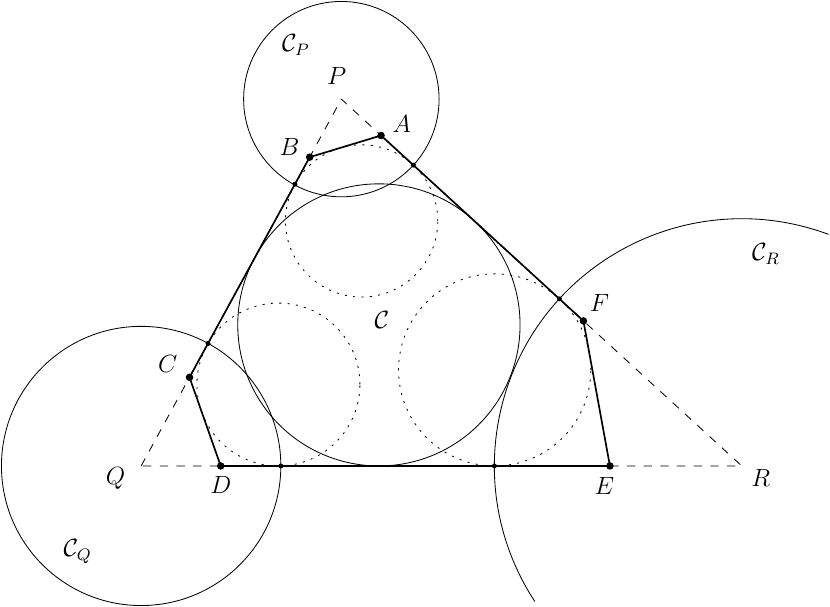}
	\caption{\small{Proof of Lemma~\ref{lem:3pairs}}.}
	\label{fig:3pairs2}
\end{figure}

\begin{proof}
For $t\ge 0$, let $B(t),C(t)\in BC$, $D(t),E(t)\in DE$, and $F(t),A(t)\in FA$ be the six points such that 
$\seg{AA(t)}=\seg{BB(t)}=\seg{CC(t)}=\seg{DD(t)}=\seg{EE(t)}=\seg{FF(t)}=t$, and the hexagon $A(t)B(t)C(t)D(t)E(t)F(t)$
is convex and satisfies the same conditions as $ABCDEF$. Let $t^*$ denote the maximum possible value of $t$.
Let $\mathcal{E}_P(t)$ be the disk whose boundary is the excircle of the triangle $\Delta PB(t)A(t)$ that
is contained in $\Delta PQR$. Let $\mathcal{C}_P(t)$ denote the disk with center $P$ and
radius $\seg{PT(t)}$, where $T(t)$ denotes the point of tangency between $\mathcal{E}_P(t)$ and $PQ$ (see Figure~\ref{fig:3pairs2}
for the case $t=0$).
Analogously, we define the disk $\mathcal{C}_Q(t)$ centered at $Q$, and the disk $\mathcal{C}_R(t)$ centered at $R$.
Since $D_{AB}=D_{A(0)B(0)}$ is contained in $\mathcal{C}_P=\mathcal{C}_P(0)$ (Lemma~\ref{lem:1}),
to prove statement~(a) it suffices to prove that $\mathcal{C}_P$ and $D_{DE}=D_{D(0)E(0)}$ do not intersect,
which is equivalent to proving that
\begin{equation}\label{eq20}
	\seg{PM_{DE}} ~>~ \seg{PT} + \frac{\seg{DE}}{2},
\end{equation}
where $T=T(0)$, $M_{D(t)E(t)}$ is the midpoint of $D(t)E(t)$, and $M_{DE}=M_{D(0)E(0)}$.
For every $t\in[0,t^*]$, observe that 
$\seg{PM_{D(t)E(t)}}=\seg{PM_{D(0)E(0)}}=\seg{PM_{DE}}$ since $M_{D(t)E(t)}=M_{D(0)E(0)}=M_{DE}$.
Furthermore, 
$\seg{D(t)E(t)}=\seg{D(0)E(0)}-2t=\seg{DE}-2t$.
Given a triangle with sides $\mathsf{a}$, $\mathsf{b}$, and $\mathsf{c}$, and the excircle $\mathcal{E}$ of the triangle
tangent to $\mathsf{c}$ and to the extensions of $\mathsf{a}$ and $\mathsf{b}$, the segments of such extensions
connecting the common vertex of $\mathsf{a}$ and $\mathsf{b}$ with one point of tangency with $\mathcal{E}$ have
length $(1/2)(\seg{\mathsf{a}}+\seg{\mathsf{b}}+\seg{\mathsf{c}})$. Then, since $\mathcal{E}_P(t)$ is an excircle
of $\Delta PA(t)B(t)$, we also have
\begin{eqnarray*}
	\seg{PT(t)}  & = & \frac{1}{2}\left(\seg{PB(t)}+\seg{B(t)A(t)}+\seg{A(t)P}\right) \\
	         & = & \frac{1}{2}\left(\seg{PB(0)}+t +\seg{B(t)A(t)}+\seg{A(0)P}+t\right) \\
	         & = & \frac{1}{2}\left(\seg{PB} +\seg{B(t)A(t)}+\seg{AP}\right) + t.
\end{eqnarray*}
Consider the function $G:[0,t^*]\rightarrow\mathbb{R}$ defined as follows:
\[
	G(t) ~=~ \seg{PM_{D(t)E(t)}} - \seg{PT(t)} - \frac{\seg{D(t)E(t)}}{2},
\]
which satisfies
\begin{eqnarray*}
	G(t) & = & \seg{PM_{D(t)E(t)}} - \seg{PT(t)} - \frac{\seg{D(t)E(t)}}{2} \\
	     & = & \seg{PM_{DE}} - \frac{1}{2}\left(\seg{PB} +\seg{B(t)A(t)}+\seg{AP}\right) - t - \frac{\seg{DE}}{2} + t \\
	     & = & \seg{PM_{DE}} - \frac{1}{2}\left(\seg{PB} +\seg{B(t)A(t)}+\seg{AP}\right) - \frac{\seg{DE}}{2}.
\end{eqnarray*}
Since the function $\seg{B(t)A(t)}$ is increasing in the range $t\in[0,t^*]$, 
we have that
\[
	\seg{PM_{DE}} - \seg{PT} -  \frac{\seg{DE}}{2} ~=~ G(0) ~\ge~ G(t^*) 
		~=~ \seg{PM_{D(t^*)E(t^*)}} - \seg{PT(t^*)} -  \frac{\seg{D(t^*)E(t^*)}}{2}.
\]
Then, to prove inequation~\eqref{eq20} and then statement~(a), it suffices to show that $G(t^*)>0$, which is equivalent to showing
that $\mathcal{C}_P(t^*)$ and $D_{D(t^*)E(t^*)}$ do not intersect. Analogously, to prove statement~(b) it suffices to show
that $\mathcal{C}_Q(t^*)$ and $D_{F(t^*)A(t^*)}$ do not intersect, and to prove statement~(c) it suffices to show
that $\mathcal{C}_R(t^*)$ and $D_{B(t^*)C(t^*)}$ do not intersect.

\begin{figure}[t]
	\centering
	\includegraphics[scale=0.75,page=13]{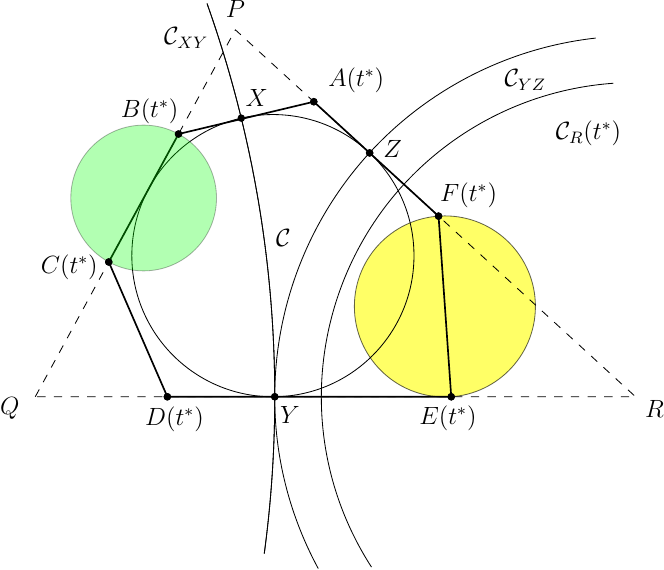}
	\caption{\small{Proof of Lemma~\ref{lem:3pairs}}.}
	\label{fig:3pairs3}
\end{figure}

Observe from the defintion of $t^*$ that $\mathcal{C}$ is tangent to at least one of the segments
$A(t^*)B(t^*)$, $C(t^*)D(t^*)$, and $E(t^*)F(t^*)$. Assume w.l.o.g.\ that $\mathcal{C}$ is tangent
to $A(t^*)B(t^*)$. Let $X$, $Y$, and $Z$ be the points of tangency between $\mathcal{C}$ and
the segments $A(t^*)B(t^*)$, $D(t^*)E(t^*)$, and $F(t^*)A(t^*)$, respectively.
Further assume w.l.o.g.\ that the line $\ell(Q,R)$ is horizontal, and either
the lines $\ell(A(t^*),B(t^*))$ and $\ell(Q,R)$ are parallel or the point $\ell(A(t^*),B(t^*))\cap \ell(Q,R)$ is to the left of $Q$ 
(see Figure~\ref{fig:3pairs3} in which the latter case occurs). In the former case, let
$\mathcal{C}_{XY}$ denote the halfplane with the points in or to the left of the vertical line $\ell(X,Y)$.
In the latter case, let $\mathcal{C}_{XY}$ denote the disk centered at $\ell(A(t^*),B(t^*))\cap \ell(Q,R)$
whose boundary contains $X$ and $Y$. Let $\mathcal{C}_{YZ}$ denote the disk with center $R$ and radius
$\seg{RY}=\seg{RZ}$. By Lemma~\ref{lem:1} and construction, we have both
\[
	D_{B(t^*)C(t^*)} \subset D_{B(t^*)Q} \subset \mathcal{C}_{XY} ~~\text{and}~~
	D_{E(t^*)F(t^*)} \subset \mathcal{C}_R(t^*) \subseteq \mathcal{C}_{YZ},
\]
which implies that $D_{B(t^*)C(t^*)}$ and $D_{E(t^*)F(t^*)}$ do not intersect. Hence, statement~(c)
is satisfied and the lemma follows.
\end{proof}

\begin{figure}[t]
	\centering
	\subfloat[]{
		\includegraphics[scale=0.7,page=16]{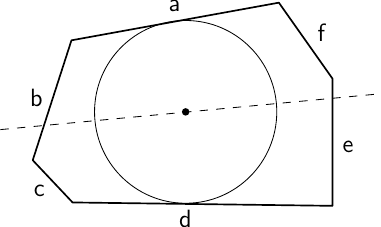}
		\label{fig:6disks1}
	}\hspace{0.2cm}
	\subfloat[]{
		\includegraphics[scale=0.7,page=17]{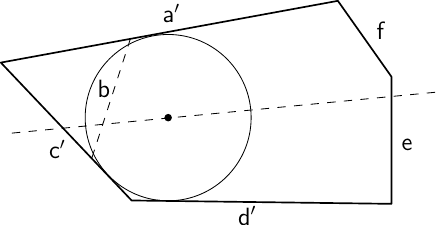}
		\label{fig:6disks2}
	}\hspace{0.2cm}
	\subfloat[]{
		\includegraphics[scale=0.7,page=18]{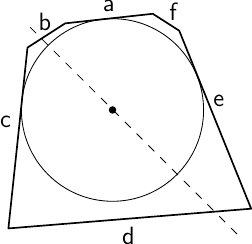}
		\label{fig:6disks3}
	}\\
	\subfloat[]{
		\includegraphics[scale=0.7,page=19]{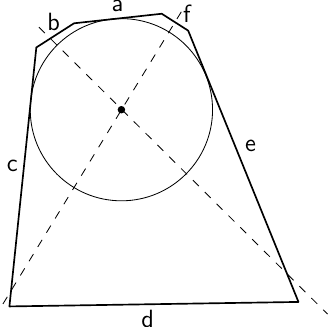}
		\label{fig:6disks4}
	}\hspace{0.2cm}
	\subfloat[]{
		\includegraphics[scale=0.7,page=20]{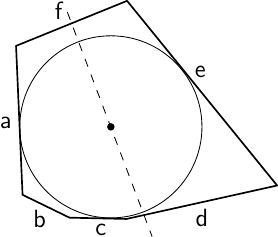}
		\label{fig:6disks5}
	}\hspace{0.2cm}
	\subfloat[]{
		\includegraphics[scale=0.7,page=21]{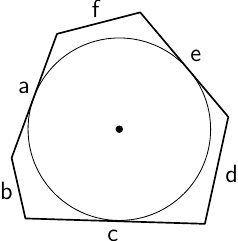}
		\label{fig:6disks6}
	}
	\caption{\small{
		Proof of Lemma~\ref{lem:no3cycles}.
	}}
	\label{fig:6disks}
\end{figure}

\begin{proof}[Proof of No-3-Cycles lemma (Lemma~\ref{lem:no3cycles})]
By extending $\mathsf{a}$, $\mathsf{b}$, $\mathsf{c}$, $\mathsf{d}$, $\mathsf{e}$, $\mathsf{f}$, we can consider that $\mathsf{a}$, $\mathsf{b}$, $\mathsf{c}$, $\mathsf{d}$, $\mathsf{e}$, $\mathsf{f}$ are the sides of an (unbounded) convex 6-gon $P_6$.
The proof is split into several cases.
Suppose that there exists a disk contained in $P_6$ and tangent to two opposed sides, say w.l.o.g.\ that the disk
is tangent to $\mathsf{a}$ and $\mathsf{d}$. Further assume w.l.o.g.\ that $\mathsf{d}$ is horizontal,
the bisector of $\mathsf{a}$ and $\mathsf{d}$
intersects the side $\mathsf{e}$, and either the lines $\ell(\mathsf{a})$ and $\ell(\mathsf{d})$ are parallel
or the point $\ell(\mathsf{a})\cap \ell(\mathsf{d})$ is to the left of $\mathsf{d}$ (see Figure~\ref{fig:6disks1}).
Using Lemma~\ref{lem:abcx} with a disk tangent
to extensions $\mathsf{a}'$, $\mathsf{c}'$, and $\mathsf{d}'$ of $\mathsf{a}$, $\mathsf{c}$, and $\mathsf{d}$,
respectively, it follows that
$D_{\mathsf{c}}\cap D_{\mathsf{f}}=\emptyset$ (see Figure~\ref{fig:6disks2}). 

The next cases use similar arguments (i.e.\ applying Lemma~\ref{lem:abcx}).
If there does not exist any disk contained in $P_6$ and tangent to two opposed sides,
then there must exist a disk contained in $P_6$ and tangent to three pairwise non-consecutive sides.
Assume w.l.o.g.\ that such a disk is tangent to  $\mathsf{a}$, $\mathsf{c}$, and $\mathsf{e}$. 
If the lines $\ell(\mathsf{a})$, $\ell(\mathsf{c})$, and $\ell(\mathsf{e})$ do not bound a triangle 
that contains $P_6$ (see Figure~\ref{fig:6disks3}), then we proceed as follows. Assume w.l.o.g.\
that either $\ell(\mathsf{c})$ and $\ell(\mathsf{e})$ are parallel or the point $\ell(\mathsf{c}) \cap \ell(\mathsf{e})$
is separated from $P_6$ by $\ell(\mathsf{a})$, as in Figure~\ref{fig:6disks3}. 
If the bisector of $\mathsf{a}$ and $\mathsf{c}$ intersects $\mathsf{d}$ (see Figure~\ref{fig:6disks3}), then $D_{\mathsf{b}} \cap D_{\mathsf{e}} = \emptyset$ by Lemma~\ref{lem:abcx}.
Analogously, if the bisector of $\mathsf{a}$ and $\mathsf{e}$ intersects $\mathsf{d}$, then $D_{\mathsf{c}} \cap D_{\mathsf{f}} = \emptyset$.
Suppose now that neither the bisector of $\mathsf{a}$ and $\mathsf{c}$ intersects $\mathsf{d}$, nor 
the bisector of $\mathsf{a}$ and $\mathsf{e}$ intersects $\mathsf{d}$ (see Figure~\ref{fig:6disks4}).
Then, we have that $D_\mathsf{a}\cap D_\mathsf{d}=\emptyset$, by Lemma~\ref{lem:quad}.
Otherwise, if the lines $\ell(\mathsf{a})$, $\ell(\mathsf{c})$, and $\ell(\mathsf{e})$ do bound a triangle 
that contains $P_6$ (see Figure~\ref{fig:6disks5}), then we proceed as follows. If the bisector of 
$\mathsf{a}$ and $\mathsf{e}$ does not intersect $\mathsf{c}$ (see Figure~\ref{fig:6disks5}), 
say w.l.o.g.\ that it intersects $\mathsf{d}$, then we have $D_\mathsf{c}\cap D_\mathsf{f}=\emptyset$.
Symmetric arguments can be given if the bisector of 
$\mathsf{a}$ and $\mathsf{c}$ does not intersect $\mathsf{e}$, or
the bisector of 
$\mathsf{c}$ and $\mathsf{e}$ does not intersect $\mathsf{a}$.
Otherwise, if the bisector of 
each two sides among $\mathsf{a}$, $\mathsf{c}$, and $\mathsf{e}$
intersects the third one (see Figure~\ref{fig:6disks6}), then
$D_\mathsf{a}\cap D_\mathsf{d}=\emptyset$, or $D_\mathsf{b}\cap D_\mathsf{e}=\emptyset$,
or $D_\mathsf{c}\cap D_\mathsf{f}=\emptyset$, by Lemma~\ref{lem:3pairs}. All the cases are covered, and the lemma follows.
\end{proof}

\section{Part III: Proof of Main theorem (Theorem~\ref{theo:main})}\label{sec:partIII}

Given an (unbounded) convex polygon and two sides $\mathsf{a}$ and $\mathsf{b}$ of it, we define the 
segment (or halfline) $\mathsf{a}|\mathsf{b}$ in the case where $\mathsf{a}$ and $\mathsf{b}$ are consecutive sides, or both are the two halfline sides of the polygon when it is unbounded, as follows: If $\mathsf{a}$ and $\mathsf{b}$ are consecutive segments, then $\mathsf{a}|\mathsf{b}$ is the diagonal of the polygon connecting an endpoint of $\mathsf{a}$ with an endpoint of $\mathsf{b}$. If $\mathsf{a}$ is a halfline and $\mathsf{b}$ is a segment, then $\mathsf{a}|\mathsf{b}$ is the halfline contained in the polygon, parallel to $\mathsf{a}$, and with apex the vertex of $\mathsf{b}$ not in common with $\mathsf{a}$. If both $\mathsf{a}$ and $\mathsf{b}$ are halflines because the polygon is unbounded, then $\mathsf{a}|\mathsf{b}$ is the segment (i.e.\ diagonal) that connects the two endpoints of $\mathsf{a}$ and $\mathsf{b}$.

\begin{lemma}\label{lem:a|b}
Let $P_n$ be an (unbounded) convex $n$-gon, $n\ge 4$, and $\mathsf{a}$ and $\mathsf{b}$ two sides of $P_n$ such that the segment (or halfline) $\mathsf{a}|\mathsf{b}$ is defined. Let $\mathsf{c}$ be another side of $P_n$ such that $D_\mathsf{c}$ intersects both $D_\mathsf{a}$ and $D_\mathsf{b}$. Then, $D_\mathsf{c}$ also intersects $D_{\mathsf{a}|\mathsf{b}}$.
\end{lemma}

\begin{figure}[t]
	\centering
	\subfloat[]{
		\includegraphics[scale=0.9,page=27]{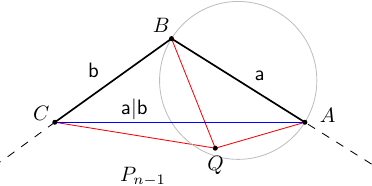}
		\label{fig:ab1}
	}\hspace{0.1cm}
	\subfloat[]{
		\includegraphics[scale=0.9,page=35]{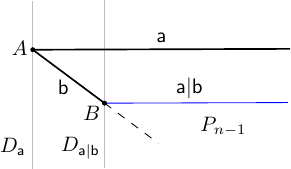}
		\label{fig:ab4}
	}\hspace{0.1cm}
	\subfloat[]{
		\includegraphics[scale=0.9,page=28]{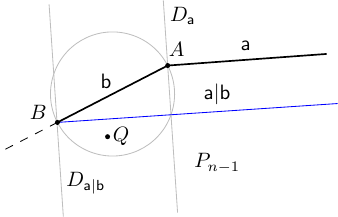}
		\label{fig:ab2}
	}\\
	\subfloat[]{
		\includegraphics[scale=0.9,page=36]{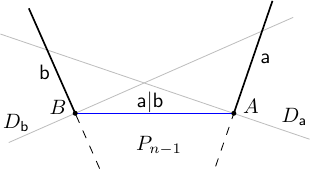}
		\label{fig:ab5}
	}\hspace{0.1cm}
	\subfloat[]{
		\includegraphics[scale=0.9,page=29]{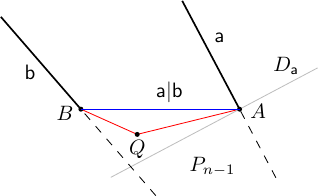}
		\label{fig:ab3}
	}
	\caption{\small{
		Proof of Lemma~\ref{lem:a|b}.
	}}
	\label{fig:a|b}
\end{figure}

\begin{proof}
Let $R_{\mathsf{a}|\mathsf{b}}$ be the convex region bounded by $\mathsf{a}$, $\mathsf{b}$, and $\mathsf{a}|\mathsf{b}$, and $P_{n-1}=P_n\setminus R_{\mathsf{a}|\mathsf{b}}$ the (possibly unbounded) convex $(n-1)$-gon resulting from removing $R_{\mathsf{a}|\mathsf{b}}$ from $P_n$. To prove the lemma, it suffices to show the following statement: every point $Q$ in $D_{\mathsf{a}}\cap P_{n-1}$, or $D_{\mathsf{b}}\cap P_{n-1}$, is also in $D_{\mathsf{a}|\mathsf{b}}$.
Assume that $\mathsf{a}$ and $\mathsf{b}$ are segments, so that $\mathsf{a}$ has endpoints $A$ and $B$, $\mathsf{b}$ has endpoints $B$ and $C$, and $\mathsf{a}|\mathsf{b}=AC$. Let $Q$ be a point in $D_{\mathsf{a}}\cap P_{n-1}$ (see Figure~\ref{fig:ab1}). Then, $\angle AQB\ge \pi/2$ by Thales' theorem. Then, we have $\angle AQC>\angle AQB\ge \pi/2$, which implies that $Q$ is also in $D_{\mathsf{a}|\mathsf{b}}$ by Thales' theorem. Analogulsy, if $Q$ is in $D_{\mathsf{b}}\cap P_{n-1}$, then it is also in $D_{\mathsf{a}|\mathsf{b}}$. 
Assume now that $\mathsf{a}$ is a halfline and $b$ is a segment, where $A$ is the apex of $\mathsf{a}$, and $\mathsf{b}$ has endpoints $A$ and $B$ (see Figure~\ref{fig:ab4} and Figure~\ref{fig:ab2}). In this case, $D_{\mathsf{a}|\mathsf{b}}$ is the halfplane containing $\mathsf{a}$ and bounded by the line through $B$ perpendicular to $\mathsf{a}|\mathsf{b}$. 
If $D_{\mathsf{a}|\mathsf{b}}\subseteq D_{\mathsf{a}}$ (see Figure~\ref{fig:ab4}), then 
$P_{n-1}\subset D_{\mathsf{a}|\mathsf{b}}$ and the statement trivially follows.
Otherwise, if $D_{\mathsf{a}}\subset D_{\mathsf{a}|\mathsf{b}}$ (see Figure~\ref{fig:ab2}), then
$D_{\mathsf{b}}\cap P_{n-1}\subset D_{\mathsf{a}|\mathsf{b}}$, and the statement follows. 
Finally, assume that both $\mathsf{a}$ and $\mathsf{b}$ are halflines, with $A$ the apex of $\mathsf{a}$, and $B$ the apex of $\mathsf{b}$ (see Figure~\ref{fig:ab5} and Figure~\ref{fig:ab3}). 
If neither $D_{\mathsf{a}}$ contains $\mathsf{b}$ nor $D_{\mathsf{b}}$ contains $\mathsf{a}$ (see Figure~\ref{fig:ab5}),
then the statement trivially follows. Otherwise, assume w.l.o.g.\ that $D_{\mathsf{a}}$ contains $\mathsf{b}$ (see Figure~\ref{fig:ab3}). Let $Q$ be a point in $(D_{\mathsf{a}}\cup D_{\mathsf{b}}) \cap P_{n-1}$, and note that $\angle AQB\ge\pi/2$ because the boundary of $D_{\mathsf{a}}$ is perpendicular to $\mathsf{a}$, and $\mathsf{a}$ and $\mathsf{b}$ are the halflines among the sides of the unbounded $P_n$. Then, $Q$ belongs to $D_{\mathsf{a}|\mathsf{b}}$ by Thales' theorem,
showing that the statement is true.
\end{proof}

\begin{figure}[t]
	\centering
	\subfloat[]{
		\includegraphics[scale=0.75,page=30]{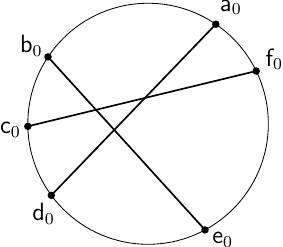}
		\label{fig:main1}
	}\hspace{0.2cm}
	\subfloat[]{
		\includegraphics[scale=0.75,page=33]{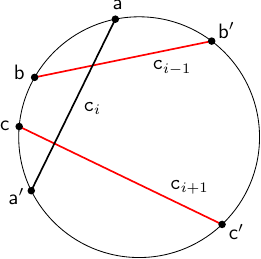}
		\label{fig:main4}
	}\subfloat[]{
		\includegraphics[scale=0.75,page=34]{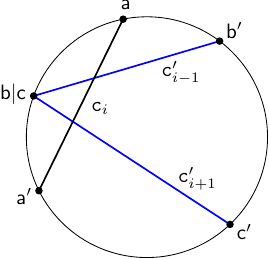}
		\label{fig:main5}
	}\\
	\subfloat[]{
		\includegraphics[scale=0.75,page=31]{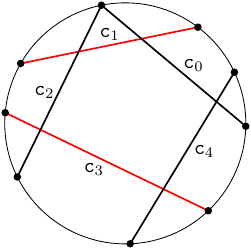}
		\label{fig:main2}
	}\hspace{0.2cm}
	\subfloat[]{
		\includegraphics[scale=0.75,page=32]{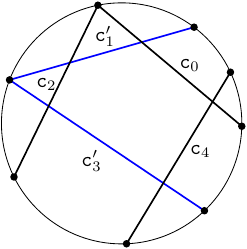}
		\label{fig:main3}
	}
	\caption{\small{
		Proof of Theorem~\ref{theo:main}.
	}}
	\label{fig:main}
\end{figure}

\begin{proof}[Proof of Main theorem (Theorem~\ref{theo:main})]
Let $P_n$ be a convex $n$-gon with $n\ge 3$. Let $G=\langle V,E \rangle$ be the intersection graph of the side disks of $P_n$, and $G_{\mathtt{c}}=\langle V_{\mathtt{c}}, E_{\mathtt{c}}\rangle$ the intersection graph of the chords in the circular embedding of $G$. If $n=3,4$, then $G$ is trivially planar. Thus, assume $n\ge 5$.
Suppose that $G_{\mathtt{c}}$ has a 3-length cycle, made of three pairwise intersecting chords, induced by six sides $\mathsf{a}_0$, $\mathsf{b}_0$, $\mathsf{c}_0$, $\mathsf{d}_0$, $\mathsf{e}_0$, $\mathsf{f}_0$ of $P_n$. Assume w.l.o.g.\ that these sides appear in this order counter-clockwise along the boundary of $P_n$ (see Figure~\ref{fig:main1}). Some (or all) of $\mathsf{a}_0$, $\mathsf{b}_0$, $\mathsf{c}_0$, $\mathsf{d}_0$, $\mathsf{e}_0$, $\mathsf{f}_0$ can be extended to obtain the sides $\mathsf{a}\supseteq \mathsf{a}_0$, $\mathsf{b}\supseteq \mathsf{b}_0$, $\mathsf{c}\supseteq \mathsf{c}_0$, $\mathsf{d}\supseteq \mathsf{d}_0$, $\mathsf{e}\supseteq \mathsf{e}_0$, $\mathsf{f}\supseteq \mathsf{f}_0$ of a possibly unbounded convex 6-gon. By Lemma~\ref{lem:no3cycles}, we have that $D_\mathsf{a}\cap D_\mathsf{d}=\emptyset$, $D_\mathsf{b}\cap D_\mathsf{e}=\emptyset$, or $D_\mathsf{c}\cap D_\mathsf{f}=\emptyset$. This implies that $D_{\mathsf{a}_0}\cap D_{\mathsf{d}_0}=\emptyset$, $D_{\mathsf{b}_0}\cap D_{\mathsf{e}_0}=\emptyset$, or $D_{\mathsf{c}_0}\cap D_{\mathsf{f}_0}=\emptyset$. Hence, 3-length cycles do not exist in $G_{\mathtt{c}}$ by contradiction. 
Let $k\ge 5$ and $c=\langle \mathtt{c}_0,\mathtt{c}_1,\ldots,\mathtt{c}_{k-1},\mathtt{c}_0\rangle$ a minimal cycle of length $k$ in $G_{\mathtt{c}}$, where \DEF{minimal} means that no proper subset of $\{\mathtt{c}_0,\mathtt{c}_1,\ldots,\mathtt{c}_{k-1}\}$ form a cycle. Assume that $\mathtt{c}_0,\mathtt{c}_1,\ldots,\mathtt{c}_{k-1}$ are sorted counter-clockwise (as in Figure~\ref{fig:main2} for $k=5$), and that they define a set of $t$ endpoints, $k\le t \le 2k$.
These endpoints correspond to side disks, and then sides, of $P_n$. Extending some, or all, of such sides we obtain a possibly unbounded convex $t$-gon $P_t$. We have two cases: $t>k$ and $t=k$.
Suppose that $t>k$. In this case, we can select a chord $\mathtt{c}_i$ such that the chords $\mathtt{c}_{i-1}$ and $\mathtt{c}_{i+1}$ that intersect $\mathtt{c}_i$ do not share any endpoint, where subindices are taken modulo $k$ (see Figure~\ref{fig:main4}). Let $\mathsf{a}$ and $\mathsf{a}'$ be the sides of $P_t$ that correspond to the endpoints of $\mathtt{c}_i$, let $\mathsf{b}$ and $\mathsf{b}'$ be the sides that correspond to the endpoints of $\mathtt{c}_{i-1}$, and let $\mathsf{c}$ and $\mathsf{c}'$ be the sides that correspond to the endpoints of $\mathtt{c}_{i+1}$; so that $\mathsf{a}$, $\mathsf{b}$, $\mathsf{c}$, $\mathsf{a}'$ are in this order counter-clockwise along the boundary of $P_t$. Observe that $\mathsf{b}|\mathsf{c}$ is defined, and let $R_{\mathsf{b}|\mathsf{c}}$ be the convex region bounded by $\mathsf{b}$, $\mathsf{c}$, and $\mathsf{b}|\mathsf{c}$, and let $P_{t-1}=P_t\setminus R_{\mathsf{b}|\mathsf{c}}$. For every chord different from $\mathtt{c}_{i-1}$ and $\mathtt{c}_{i+1}$ in the cycle $c$, and sides $\mathsf{z}$ and $\mathsf{z}'$ of $P_{t-1}$ corresponding to its endpoints, we still have in $P_{t-1}$ that $D_{\mathsf{z}}\cap D_{\mathsf{z}'}\neq \emptyset$. Furthermore, for the sides $\mathsf{b}'$ and $\mathsf{c}'$, also of $P_{t-1}, $we have both $D_{\mathsf{b}|\mathsf{c}}\cap D_{\mathsf{b}'}\neq \emptyset$ and $D_{\mathsf{b}|\mathsf{c}}\cap D_{\mathsf{c}'}\neq \emptyset$, by Lemma~\ref{lem:a|b}. This means that in the intersection graph of the chords in the circular embedding of the side disks of $P_{t-1}$ there exists a cycle of length $k$, but the chords of the cycle define a set of endpoints of precisely one less element, that is, $t-1$ endpoints (see the transition from Figure~\ref{fig:main2} to Figure~\ref{fig:main3}). Using this transition from $P_t$ to $P_{t-1}$, we can assume $t=k$ from the beggining and then for every $i\in\{0,1,\ldots,k-1\}$ we have that $\mathtt{c}_{i-1}$ and $\mathtt{c}_{i+1}$ share an endpoint. This condition implies that in $P_t$ every side disk defines at least two chords, which contradicts Lemma~\ref{lem:atmost1}. Hence, the graph $G_{\mathtt{c}}$ is bipartite since it cannot contain cycles of odd length, which implies that $G$ is planar by Theorem~\ref{theo:hamiltonian-planar}.
\end{proof}

\small

\section*{Acknowledgements}

We wish to thank Oswin Aichholzer, Ruy Fabila-Monroy, Thomas Hackl, Tillmann Miltzow, 
Christian Rubio, Eul\`alia Tramuns, Birgit Vogtenhuber, and Frank Duque for inspiring discussions on this topic.
First author was supported by Projects MTM2012-30951 and DGR 2014SGR46.
Second author was supported by projects CONICYT FONDECYT/Iniciaci\'on 11110069 (Chile), 
and Millennium Nucleus Information and Coordination in Networks ICM/FIC RC130003 (Chile).

\small

\bibliographystyle{abbrv}
\bibliography{refs}

\end{document}